\documentclass[final]{amsart}
\usepackage[utf8]{inputenc} % set input encoding (not needed with XeLaTeX)

\usepackage{graphicx} % support the \includegraphics command and options

\usepackage[skip=2pt,font=scriptsize]{caption}

\usepackage[foot]{amsaddr} %package that puts the authors addresses on the title page as a footnote (when using AMSART document class only).
\usepackage{verbatim} % adds environment for commenting out blocks of text & for better verbatim
\usepackage{amssymb}
\usepackage{amsthm}
\usepackage{amsmath}
\usepackage{amsfonts}
\usepackage[noadjust]{cite}
\usepackage{graphicx}
\usepackage{soul}

\usepackage{latexsym}
\usepackage{mathtools} % Mathematical tools to use with amsmath http://ftp.math.purdue.edu/mirrors/ctan.org/macros/latex/contrib/mathtools/mathtools.pdf
\usepackage{graphics}

\graphicspath{ {./images/} }

\usepackage{float} %Improved interface for floating objects
\usepackage{enumitem} % http://mirror.las.iastate.edu/tex-archive/macros/latex/contrib/enumitem/enumitem.pdf
\usepackage[english]{babel} %man­ages cul­tur­ally-de­ter­mined ty­po­graph­i­cal (and other) rules for a wide range of lan­guages
\usepackage{esint}

\usepackage{calrsfs}

    \usepackage[backref=none]{hyperref} %% This provides \href command
     \usepackage[usenames,dvipsnames]{color}
\definecolor{pennblue}{RGB}{1,31,91} % This is PENN Blue
\definecolor{pennred}{RGB}{153,0,0}	  % This is PENN Red
     \definecolor{MyDarkBlue}{rgb}{0,0.1,0.7}
     \hypersetup{pdfborder={0 0 0},colorlinks,breaklinks=true,
       urlcolor={pennblue},citecolor={pennblue},linkcolor={pennblue}
     }

\usepackage[color]{showkeys}
\usepackage{color}

\theoremstyle{plain}
\newtheorem{theorem}{Theorem}

\newtheorem{definition}[theorem]{Definition}
\newtheorem{proposition}[theorem]{Proposition}
\newtheorem{remark}[theorem]{Remark}

\newtheorem{corollary}[theorem]{Corollary}

\numberwithin{equation}{section}
\numberwithin{theorem}{section}

\allowdisplaybreaks

\title[Traveling wave solutions of PDEs via neural networks]{Traveling wave solutions of partial differential equations via neural networks}

\author[S. Cho]{Sung Woong Cho$^{*,1}$}
\address{$^*$Pohang University of Science and Technology, Pohang 37673, Republic of Korea. \href{mailto:swcho95kr@gmail.com}{swcho95kr@gmail.com} }
\author[H. J. Hwang]{Hyung Ju Hwang$^\dagger$}
\address{$^\dagger$Pohang University of Science and Technology, Pohang 37673, Republic of Korea. \href{mailto:hjhwang@postech.ac.kr}{hjhwang@postech.ac.kr} }
\author[H. Son]{Hwijae Son$^{\star,1}$}
\address{$^\star$Pohang University of Science and Technology, Pohang 37673, Republic of Korea. \href{mailto:son9409@postech.ac.kr}{son9409@postech.ac.kr} }

\AtBeginDocument{%
   \def\MR#1{}
}

\begin{document}

%\keywords{collisional Kinetic Theory}
%\subjclass[2000]{}

\date{\today}
%\date{\today; \Red{(DRAFT)}}
% \date{May 8, 2019}
% \date{}

\let\thefootnote\relax\footnotetext{
	$^\dagger$ corresponding author\\
	$^1$ equal contribution}
\addtocounter{footnote}{-1}\let\thefootnote\svthefootnote

\begin{abstract}
This paper focuses on how to approximate traveling wave solutions for various kinds of partial differential equations via artificial neural networks. A traveling wave solution is hard to obtain with traditional numerical methods when the corresponding wave speed is unknown in advance. We propose a novel method to approximate both the traveling wave solution and the unknown wave speed via a neural network and an additional free parameter. We proved that under a mild assumption, the neural network solution converges to the analytic solution and the free parameter accurately approximates the wave speed as the corresponding loss tends to zero for the Keller--Segel equation. We also demonstrate in the experiments that reducing loss through training assures an accurate approximation of the traveling wave solution and the wave speed for the Keller--Segel equation, the Allen--Cahn model with relaxation, and the Lotka--Volterra competition model.
\end{abstract}

% set the depth for the table of contents (0-2)
\setcounter{tocdepth}{2}
%\chapter is level 0
%\section is level 1
%\subsection is level 2
%\subsubsection is level 3
%\paragraph is level 4
%\subparagraph is level 5

\maketitle
\thispagestyle{empty}

\section{Introduction} \label{sec1}

\subsection{Motivation}
In this paper, we propose a novel method for approximating traveling wave solutions via deep neural networks. Traveling wave solution, a special form of the particular solutions of partial differential equations (PDEs) has been studied extensively. For several equations, in the case that the boundary condition consists of two different equilibrium points of the system, an interval for the wave speed where a traveling wave solution exists has been demonstrated (see, \cite{salako2016spreading, bramburger2020exact, chen2017existence}). The authors in \cite{larson1978transient, li2011asymptotic, wu2017uniqueness} discussed that even if we add a small perturbation to the traveling wave profile, it converges to the original shape. Furthermore, it is known for the Keller--Segel equation and the Lotka--Volterra competition model that a unique traveling wave solution with a unique wave speed exists up to translation (see, \cite{kan1995parameter, li2012steadily}). 

\par Although finding a traveling wave solution seems like a relatively simple ODE problem, approximating a numerical solution is not a self-evident process when the wave speed is unknown in advance. There have been several attempts to numerically approximate the wave speed by finding a new variable that has a monotone dependency on the wave speed in \cite{lattanzio2016analytical, mansour2008traveling}, but there is no theoretical evidence to guarantee the convergence to the wave speed.

\par An artificial neural network is a natural candidate for finding a traveling wave solution since it can easily model the dependency of the solution to an ansatz variable which commonly appears in the traveling wave literature. Furthermore, the universal approximation property of neural networks suggests the possibility of approximating solutions of the partial differential equations. By penalizing a neural network to satisfy given PDEs, one can guarantee the convergence of the neural network to an actual solution using the energy estimate method (see, \cite{sirignano2018dgm, hwang2020trend, NHM}).

\par In this work, we propose a novel method that simultaneously approximates the traveling wave solution of given PDEs and the wave speed. We employ fully-connected neural networks to approximate the solutions of PDEs and an additional free parameter to approximate the wave speed. We prove the convergence of both neural networks and the free parameter to the analytic solutions and the actual wave speed, respectively, for the Keller--Segel model. Moreover, the experimental results show that our estimated speeds agree with the analytic results of the Keller--Segel equation, the Lotka--Volterra equation, and the Allen--Cahn model for various kinds of parameter settings.

\subsection{Related Works}

Early studies focused on finding an explicit form of the traveling wave solution. Assuming a specific functional form of solutions (e.g. a rational function with numerator and denominator of sums of exponential functions), they solved some well-posed problems (see, \cite{he2006exp, malfliet1992solitary, wazwaz2007tanh}). In \cite{rosu2005supersymmetric, wang2005new}, multi-dimensional traveling wave solutions were constructed by using the solution of the Riccati equation. The authors in \cite{qin2018rogue} proposed one to find a simpler sufficient condition for the solutions of the original equation by factoring the differential operator.

\par 
There are several works that try to numerically approximate the solution, assuming a specific functional form, such as a polynomial function, for the nonlinear term. For the non-linear Allen--Cahn equation, assuming a fractional power series solution, the solution is iteratively calculated in \cite{tariq2017new} using the condition that the coefficients must meet for the residual term to be zero. Assuming that the nonlinear term is an admonian polynomial, an integral iteration guarantees the convergence after combining the initial conditions and the governing equation by the Laplace transform for the fractional Whitham-Broer-Kaup equation (see, \cite{ali2018numerical}). For the Korteweg–de Vries equation, the authors in \cite{kaya2004exact} applied the basis extension method for the solution under the assumption that the nonlinear terms are admonian polynomials.

\par Stability of the traveling wave solutions has been also actively studied. For an equation containing a nonlinear fisher term, the authors in \cite{hagan1982traveling, larson1978transient} showed that when the initial data has exponential decay, it converges to a traveling wave solution with a certain wave speed. For the reaction-advection-diffusion equation in \cite{wang2007existence}, a unique traveling wave solution exists and all solutions converge to that traveling wave. Given the boundary conditions for the classical Keller--Segel equation, a unique traveling wave solution exists and its stability against small initial perturbations has been demonstrated in \cite{li2012steadily}. Stability for the traveling wave solution with a speed above a certain value in a multi-type SIS non-localized epidemic model was provided in \cite{wu2017uniqueness}. The aforementioned theoretical results for the stability also contributed to the numerical method for the traveling wave solution where the exact speed values are unknown. The authors in \cite{yang2020numerical} calculated the wave speed of a traveling wave solution with a globally stable equilibrium point as an endpoint. More specifically, they assume that the solution has converged to a steady state after a long time period, and observe how fast the point moves from that time.
Since it is difficult to implement the infinite domain numerically, the boundary condition was newly processed using the exponential decay rate of the traveling wave solution (see, \cite{hagstrom1986numerical}). 
Using only classical FDM, it was verified that a solution of the original equation approaches to a traveling wave solution with a specific speed in the Fisher's equation.

\par 
Several papers have introduced methods that directly estimate the value of speed. For the Allen--Cahn model with relaxation model which is a coupled equation without a diffusion term, the wave speed was estimated in \cite{lattanzio2016analytical} using a function with a monotonic dependency on the speed. 
A similar method was proposed to find the minimum wave speed at which a unique traveling wave solution exists for the reaction-diffusion chemotaxis model, by comparing the trajectories connecting equilibriums (see, \cite{mansour2008traveling}). For the Keller--Segel equation with the Fisher birth terms, the authors in \cite{bramburger2020exact} determined whether a heteroclinic orbit could leave certain regions and obtained an exact minimum wave speed. An analogous discussion was developed in \cite{chen2017existence} for the isothermal diffusion system.

\subsection{Outline of the paper} 
In Section \ref{sec2}, we introduce the models, the loss functions, and the training procedure. In Section \ref{sec3}, we cover the Keller--Segel (KS) equation derived by adding a singular term to the classical Patlak--Keller--Segel equation. We prove that our method can accurately approximate the traveling wave solution as well as the wave speed by reducing the proposed loss function. Additionally, we derive a uniform bound for the difference between a neural network solution and an analytic solution. The experiments that support our theoretical results are also presented in Section \ref{sec3}. In Section \ref{sec4} and \ref{sec5}, by simply modifying the theorems proved in Section \ref{sec3}, we apply our method to other equations, the Allen--Cahn model with relaxation and the Lotka--Volterra competition model. Finally, the article concludes in Section \ref{sec6} by introducing issues that may be addressed in the future.

%Section2--------------------------------------------------------------------
\section{Methodology} \label{sec2}
In this paper, we consider several systems of PDEs that attain traveling wave solutions. Consider a system of PDE :

\begin{align}\label{pde_system}
    \begin{split}
        u_t(t,x) &= F(u,v),  \\
        v_t(t,x) &= G(u,v),  \\
        (u,v)(t,0) &= (u_0(x),v_0(x)) \rightarrow \begin{cases}
        (u_{-}, v_{-}) \quad \text{as } x \rightarrow -\infty,\\
        (u_{+}, v_{+}) \quad \text{as } x \rightarrow +\infty, 
        \end{cases}
    \end{split}
\end{align}
where $F$, and $G$ are arbitrary differential operators. We denote the characteristics by $z = x-st$, where $s$ denotes the wave speed, and the traveling wave solutions by $U(z) = U(x-st) = u(t,x)$, and $V(z) = V(x-st) = v(t,x)$. Using this representation, we can rewrite \eqref{pde_system} as :
\begin{align}
    P(U, V; s) &= 0 \nonumber \\
    Q(U, V; s) &= 0 \nonumber \\
    (U, V)(z) &\rightarrow \begin{cases}
    (u_{-}, v_{-}) \quad \text{as } z \rightarrow -\infty,\\
    (u_{+}, v_{+}) \quad \text{as } z \rightarrow +\infty, 
    \end{cases}\nonumber
\end{align}
where P, Q are the differential operators that can be computed from \eqref{pde_system}.

In this section, we provide a detailed description of our methodology for finding approximations of the traveling wave solutions. We use two neural networks $U^{nn}$, $V^{nn}$, and a free parameter $s^{nn}$ to approximate the solutions $U, V$ and the wave speed $s$, respectively. We treat the approximation problem as an optimization problem of a properly defined loss function with respect to the parameters of $U^{nn}, V^{nn}$, and the free parameter $s^{nn}$. Previously, a universal method using a neural network was presented in \cite{lu2021deepxde} for approximating a solution of PDEs involving unknown parameters. The main differences from our method lie in that the constraints that unknown parameters must satisfy are directly reflected in the structure of the neural network and that we introduce an additional loss function to handle infinite domains. The formulation will be justified in Section \ref{sec3}.

\subsection{Neural Network Model}
The fully connected neural networks $U^{nn}$, and $V^{nn}$ take the spatio-temporal grid points $(t,x)$ as inputs, and output the approximations of $U(t,x)$, and $V(t,x)$, respectively. In this paper, a special layer is added between the input and the first hidden layer, the characteristics (or traveling wave ansatz) layer. The characteristics layer transforms a given spatio-temporal grid point $(t,x)$ to a point on the characteristic line $z = x-s^{nn}t$, where $s^{nn}$ denotes an approximator for the wave speed $s$. To explain it more precisely, we consider a neural network that consists of $L+2$ layers with $L-1$ hidden layers. The values of neurons belonging to each layer are determined by the following recurrence equation:
\begin{align} \label{recurrence}
\begin{cases}
N_{0}(t,x)=N_{0}^{(1)}(t,x)=x-s^{nn}t,\\
\displaystyle N_{l}^{(j)}=\sigma(\sum_{i=1}^{h}w_{l}^{(i, j)}N_{l-1}^{(i)}+b_l^{(j)}), \quad \text{for} \quad l=1, 2, ..., L-1,\\
\displaystyle N_{L}(t,x)=N_{L}^{(1)}(t,x)=\sum_{i=1}^{h}w_{L}^{(i,1)}N_{L-1}^{i},
\end{cases}
\end{align}
where $N_{0}(t,x)$ denotes the characteristics layer, $N_{L}^{(i)}$ denotes the $i$-${th}$ neuron of the $L$-$th$ layer, $h$ denotes the number of neurons in each layer, $\sigma$ denotes an activation function, and $w^{(i,j)}_L$, $b^{(i)}_l$ denote the weight and bias in each layer. Due to the presence of $N_0$ layer, $N_L(t,x)$ becomes a traveling wave function with a characteristic line of slope $s^{nn}$ in the $(t,x)$ plane. Therefore, our neural networks become $U^{nn}(t,x) = U^{nn}(x-s^{nn}t)$, and $V^{nn}(t,x) = V^{nn}(x-s^{nn}t)$.

Additionally, if the exact bound of the solution is known, one more activation function can be implemented to the output layer in order for the range of the neural network to meet the known bound. Then, the equation below replaces the third equation of \eqref{recurrence}.
\begin{equation}
\displaystyle N_{L}(t,x)=3(u_--u_+)S(\sum_{i=1}^{h}w_{L}^{(i,1)}N_{L-1}^{i})+(2u_+-u_-), \nonumber
\end{equation}
where $S(x)$ denotes the sigmoid function $\frac{e^{x}}{1+e^{x}}$. 

\begin{remark}
We note here that it is possible to create a network that receives one dimensional input rather than (t,x) by taking the traveling wave ansatz in the given equation in advance. However, the above method can be generalized to more complex cases, since it can be applied by slightly modifying the form of $N_0(t,x)$ when the shape of the characteristics changes or even when the shape is not determined. 
\end{remark}

\subsection{Loss functions} \label{sec2.2}
Now we define the loss functions. Firstly, we define the $L^2$ loss for the governing equation in the interior region. Since the ansatz variable $z$ lies in $\mathbb{R}$, we need to define the loss function on $\mathbb{R}$. However, it is difficult to deal with the infinite domain when training a neural network. Therefore, we truncate the real line by $[-a, a]$ for some large $a$ as in \cite{hwang2020trend}. Then the loss function for each governing equation is defined by :

\begin{align}
   Loss_{GE}^{(1)} = \int_{-a}^{a} (P(U^{nn}, V^{nn}; s^{nn}))^2 dz \approx \sum_{i} P(U^{nn}(z_i), V^{nn}(z_i); s^{nn})^2, \nonumber\\
   Loss_{GE}^{(2)} = \int_{-a}^{a} (Q(U^{nn}, V^{nn}; s^{nn}))^2 dz \approx \sum_{i} Q(U^{nn}(z_i), V^{nn}(z_i); s^{nn})^2. \nonumber
\end{align}
We then define the loss function for the governing equation by combining the losses.
\begin{align}
    Loss_{GE} = Loss^{(1)}_{GE}+Loss^{(2)}_{GE}. \nonumber
\end{align}
Since it is also difficult to impose an asymptotic boundary condition, we bypass the goal to reduce the difference between the extreme value and the value at the end of the boundary interval.
\begin{align*}
    Loss^{(1)}_{Limit} &= (U^{nn}(-a)-u_{-})^2  + (U^{nn}(a)-u_{+})^2, \\
    Loss^{(2)}_{Limit} &= (V^{nn}(-a)-v_{-})^2  + (V^{nn}(a)-v_{+})^2, \\
    Loss_{Limit} &= Loss^{(1)}_{Limit} + Loss^{(2)}_{Limit}. \nonumber
\end{align*}

In practice, the integral is approximated by the Monte--Carlo method. Using fixed points to approximate the integral is not suitable for functions that change rapidly in values. For such an intuitive reason, we uniformly sample new grid points from $[-a,a]$ to approximate the loss function for each training epoch. The iterative sampling technique is first introduced in \cite{sirignano2018dgm}.

We add the following Neumann boundary condition to more accurately estimate the wave speed (it will be further demonstrated in the next section). If the derivatives have a limit on each side, the limit must be zero. Therefore, the Neumann boundary condition is a reasonable constraint for finding the solution.
\begin{align} 
    \displaystyle Loss_{BC} = (\frac{d}{dz} U^{nn}(-a))^2 + (\frac{d}{dz} U^{nn}(a))^2 + (\frac{d}{dz} U^{nn}(-a))^2 +(\frac{d}{dz} U^{nn}(a))^2. \nonumber
\end{align}
Since the translation of the traveling wave solution becomes a solution again, we fix the solutions at a point $z=0$. Because at least one component of the solutions $(U,V)$ is a monotone function, we give the label at point $z=0$ by the mean of the limits. The loss below prevents translation so that increasing the value of $a$ has the effect of widening both sides of the domain.

\begin{align}
    \displaystyle Loss_{Trans} = (U^{nn}(0)-\frac{u_-+u_+}{2})^2, \quad or \quad  Loss_{Trans} = (V^{nn}(0)-\frac{v_-+v_+}{2})^2. \nonumber
\end{align}
The optimization process reduces the total loss created by combining all the losses defined above. We present the overall architecture in Figure \ref{fig:overall}.
\begin{align}
    Loss_{Total} = Loss_{GE} + Loss_{Limit} + Loss_{BC} + Loss_{Trans}. \nonumber
\end{align}

\begin{figure}[h]
\centering
\includegraphics[width=0.8\linewidth]{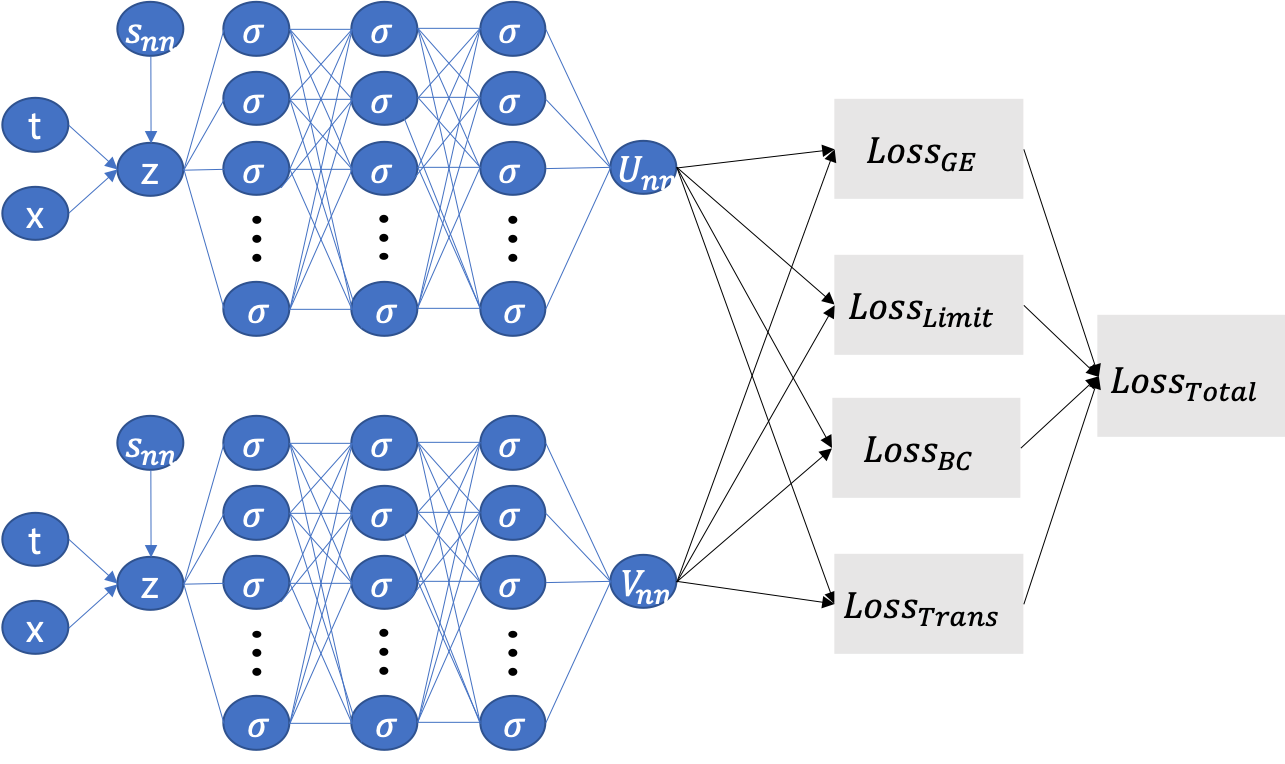} 
\caption{Pictorial description of the overall architecture.}
\label{fig:overall} 
\end{figure}

\subsection{Training}
The training procedure consists of two parts: feed-forward and back-propagation. The first step is simply feeding the input $(t,x)$ together with $s^{nn}$ through the forward path of the neural network. The second step, so called back-propagation, is to compute all the partial derivatives of the loss function with respect to $w^{(i,j)}_l$, $b^{(i)}_l$, and $s^{nn}$, and this can be easily done by Automatic Differentiation (see, \cite{paszke2017automatic} for more information). Once we compute all the partial derivatives, we optimize the loss function so that $U^{nn}$, $V^{nn}$, and $s^{nn}$ approximate the solutions and the wave speed. More specifically, we solve an optimization problem stated below:
\begin{equation}
    \underset{U^{nn}, V^{nn}, s^{nn}}{\text{minimize}} Loss_{Total} (U^{nn}, V^{nn}; s^{nn}). \nonumber
\end{equation} 

The minimization problem can be solved by a gradient based optimization technique. By using the partial derivatives, one can perform the gradient descent step by: 
\begin{align}
    w_{l}^{(i,j)} \leftarrow w_{l}^{(i,j)} - \eta \frac{\partial Loss_{Total}}{\partial w_{l}^{(i,j)}}, \nonumber\\
    b_{l}^{(i)} \leftarrow b_{l}^{(i)} - \eta \frac{\partial Loss_{Total}}{\partial b_{l}^{(i)}}, \nonumber\\
    s^{nn} \leftarrow s^{nn} - \eta \frac{\partial Loss_{Total}}{\partial s^{nn}}, \nonumber
\end{align}
where $\eta$ denotes the learning rate. In this paper, we employed ADAM as an optimizer which is one of the most popular optimizers \cite{kingma2014adam}.
%Section3--------------------------------------------------------------------
\section{Applications to the Keller--Segel Model} \label{sec3}
This section describes the details of the deep neural network used for the approximation of traveling wave solutions. Primarily, we deal with the classical Keller--Segel model with a chemotactic sensitivity term in the form of a logarithmic function which was initially devised in \cite{keller1971traveling}. The exact model is as follows.
\begin{align}
\label{KSO}
\begin{cases}
 u_t=(Du_x-\chi u c^{-1}c_x)_x, \\ 
 c_t=\epsilon c_{xx}-uc+\beta c, 
\end{cases}
\end{align}
with the boundary conditions
\begin{align}
    (u(0,x),v(0,x))=(u_0(x), v_0(x)), \quad where\quad (u_0(\pm \infty) , v_0(\pm \infty))=(u_{\pm}, v_{\pm}). \nonumber
\end{align}
Due to the difficulty of numerical computations, the singular term $c^{-1}$ was eliminated by substituting $-c_x/c=-(\ln c)_x$ to $v$ (commonly called hopf cole transformation, see, \cite{li2012steadily}). 
By imposing the traveling wave ansatz $(u,v)(x,t)=(U,V)(x-st)$, we finally derive the following ordinary differential equation below
\begin{align}
\label{3.1}
\begin{cases}
 sU_z+\chi(UV)_z+DU_{zz}=0, \\ 
 sV_z-(\varepsilon V^2 -U)_z+\varepsilon V_{zz}=0, 
\end{cases}
\end{align}
with the boundary condition $(U, V)(\pm \infty)=(u_{\pm}, v_{\pm}) $ where $u_{\pm}\ge0, v_{\pm}\le 0$. 
\\As a prerequisite for the existence of a value of s, we assume that the boundary condition is given only within the following constraints:
\begin{align}
\frac{u_{+}-u_{-}}{v_{+}-v_{-}}=\frac{\chi(u_-v_- -u_+v_+)}{\varepsilon(v_+)^2 -\varepsilon(v_-)^2 +u_- -u_+}. \nonumber
\end{align}

We refer to a proposition mentioning that the speed where the solution exists is unique, and that the solution is also unique except for the possibility of translation.
\begin{proposition}
[Proposition 2.1 in \cite{li2012steadily}] For a sufficiently small $\epsilon>0$, the solution of \eqref{3.1} satisfying the boundary condition has a monotone shock solution which is unique up to translation and satisfies $U_z<0, V_z>0$. Furthermore, the unique wave speed is explicitly evaluated as \begin{align} \label{eq3.2}
    s=-\frac{\chi v_-}{2}+\frac{1}{2}\sqrt{\chi^2 {v_-}^2+4u_{+}\chi \left[1-\varepsilon \frac{{v_+}^2-{v_-}^2}{u_+-u_-} \right]}.
\end{align}     
\end{proposition}
Another important point to note in the above theorem is the strict monotonicity of the solution, 
The range of the solution must be bounded by both extremes which are given in the boundary condition \eqref{3.1}.

\subsection{Loss Functions}
We set the value of $a$ to 200 and defined the Loss function only in [-200,200].
First, $Loss_{GE}$ is created by using the $L^2$ error of the governing equation of the Keller--Segel system.
\begin{align*}
   \displaystyle
    Loss^{(1)}_{GE} &= \int_{-a}^{a}(s^{nn}U^{nn}_z+\chi(U^{nn}V^{nn})_z+DU^{nn}_{zz})^2 dz \\ &\approx \frac{2a}{m}\sum_{i=1}^{m} (s^{nn}U^{nn}_z+\chi(U^{nn}V^{nn})_z+DU^{nn}_{zz})^2(z_i),
    \\Loss^{(2)}_{GE} &= \int_{-a}^{a}(s^{nn}V^{nn}_z-(\varepsilon (V^{nn})^2 -U^{nn})_z +\varepsilon V^{nn}_{zz})^2 dz \\ &\approx \frac{2a}{m}\sum_{i=1}^{m} (s^{nn}V^{nn}_z-(\varepsilon (V^{nn})^2 -U^{nn})_z +\varepsilon V^{nn}_{zz})^2(z_i).
\end{align*}
As a result of the experiment, it turned out that whether using $(U^{nn}(0)-\frac{u_-+u_+}{2})^2$ or $(V^{nn}(0)-\frac{v_-+v_+}{2})^2$ has nothing to do with an approximation capability. We select $(U^{nn}(0)-\frac{u_-+u_+}{2})^2$ as a translation error.
To prove the validity of an additional boundary condition, we integrate equation \eqref{3.1} so that we can obtain the following where $c_1$ and $c_2$ indicate the constants of integration. 
\begin{align}
\begin{cases}
 sU+\chi(UV)+DU_{z}=c_1, \nonumber\\ 
 sV-(\varepsilon V^2 -U)+\varepsilon V_{z}=c_2.\nonumber 
\end{cases}
\end{align}
The above equation implies that each of $DU_{z}$ and $\varepsilon V_{z}$ converges to a finite value as $z\rightarrow \infty$ or $z\rightarrow-\infty$. Since the only possible limit is zero, $Loss_{BC}$, the loss corresponding to the Neumann boundary condition, can be added in this case. 
$Loss_{Limit}$ is defined as in Section \ref{sec2}. Finally, $Loss_{Total}$ was defined as a sum of the four losses without weights. 

\subsection{Theoretical Results}
In this section, we denote the neural network solution by $U^{nn}, V^{nn}$ which are approximations of $U, V$ respectively. And the error terms, $U-U^{nn}$ and $V-V^{nn}$ are written as $E(z)$ and $F(z)$. To represent the set of functions that the neural network can approximate, we refer to the following definition and theorem in \cite{li1996simultaneous}. 
\begin{definition} For a compact set $K$ of $\mathbb{R}^n$ and positive integer $m$, we say $f\in \widehat{C}^m (K)$ if there exists an open $\Omega$ (depending on $f$) such 
that $K\subset \Omega$ and $f\in C^m(\Omega)$.

\end{definition}

\begin{theorem} \label{Thm3.3}(Li, Theorem 2.1 in \cite{li1996simultaneous}) Let $K$ be a compact subset of $\mathbb{R}$. For $m\in \mathbb{Z}_+$, if f belongs to $\widehat{C}^{m}(K))$ and $\sigma$ is any non-polynomial function in ${C}^{m}(\mathbb{R})$, then for any $\epsilon>0$, there exists a Network $\displaystyle N(x)=\sum_{i=1}^{h}c_i \sigma(w_i x +b_i)$ such that 
\begin{align*}
    ||D^{\alpha}(f)-D^{\alpha}(N)||_{L^\infty(K)}<\epsilon, \quad  \forall \alpha \in \left\{0, 1, 2, \cdots , m\right\}. 
\end{align*}
\end{theorem}
Our neural network involves the additional activation function to approximate a function with known upper and lower bounds. A similar theorem can be obtained for the approximation capability of the modified natural network model by applying the theorem above.
\begin{corollary}\label{Cor3.4} Let $S(x)$ denote the sigmoid function $\displaystyle \frac{e^x}{1+e^x}$. Suppose that function $U$ is bounded with lower bound $u_+$ and upper bound $u_-(\ne u_+)$. For $m \in \mathbb{Z}_{+}$, if $U$ belongs to $\widehat{C}^{m}(K)$ and $\sigma$ is any non-polynomial function in ${C}^{m}(\mathbb{R})$, then for any $\epsilon>0$, there exists a neural network 
\begin{equation}
    N(x)=3(u_--u_+) S(\sum_{i=1}^{h} c_i \sigma(w_{i}x+b_i))+(2u_+-u_-) \nonumber
\end{equation}
such that
\begin{equation}
    ||D^{\alpha}(U)-D^{\alpha}(N)||_{L^\infty(K)}<\epsilon, \quad  \forall \alpha \in \left\{0, 1, 2, \cdots , m\right\}. \nonumber
\end{equation}
\end{corollary}

\begin{proof}
    Let's denote $3(u_--u_+)$, $2u_+-u_-$ by $C_1, C_2$ respectively. $S, S^{-1}$ are smooth functions so that $ f:= S^{-1}\circ \left(\frac{U - C_2}{C_1}\right)$ lies in $\widehat{C}^{m}(K)$. By Theorem \ref{Thm3.3}, $f$ can be approximated by a neural network $ f_N(x)=\sum_{i=0}^{h}c_i \sigma(w_i x +b_i)$ with the property $||D^\alpha(f)-D^\alpha(f_N)||_{L^{\infty}(K)}<\epsilon$, $\forall 0\le \alpha\le m$ for some positive $\epsilon$. Let $N(x)$ be a neural network defined as $C_1 {S\circ f_N } + C_2$. Since $S$ is uniformly continuous on $K$, $||U-N||_{L^{\infty(K)}}$ can be bounded by a constant multiple of $\epsilon$. Using the triangular inequality, we derive the following.
    \begin{align}
    \nonumber
    |(S\circ f)' - (S \circ f_N)'|\le |S'\circ f - S' \circ f_N||f'|_{L^{\infty}(K)}+|f'-f_N'||S'|_{L^{\infty}(K)}.
    \end{align}
    By the uniform continuity of $f'$ and $S'$ on $K$, again $||U'-N'||_{L^{\infty(K)}}$ can be bounded by a constant multiple of $\epsilon$. Since all the terms of $D^\alpha(S \circ f)$ can be represented as products of $S^{(i)}(f)$ and $f^{(j)}$, we can derive an upper bound using a similar way above.
\end{proof}

\begin{remark}
The image of $N(x)$, $(2u_+-u_-,\hspace{0.2cm} 2u_--u_+)$, can be adjusted to different open intervals containing the interval $(u_+, u_-)$ by changing the coefficients, And, $S(x)$ can be replaced by a bounded smooth function, the inverse of which must be also smooth. Since the input value of $S(x)$ can be increased rapidly to prevent convergence of learning, it is avoided to set the range of neural networks to exactly $[u_{+}, u_{-}]$.
\end{remark}
The approximate capability of the neural network for an arbitrary function in $\widehat{C}^{m}(K)$ can be applied to solve differential equations. The following theorem states that a network function with our structure can represent a function close enough to the solution of the Keller--Segel system.
\begin{theorem} \label{3.6} For any $\epsilon >0$, there exists a sufficiently large positive integer $h$\\ such that for some constant $M>0$, if the truncated domain $[-a, a]$ contains $[-M, M]$, there exist neural networks 
\begin{equation}\displaystyle \begin{cases} U^{nn}(t, x)=3(u_--u_+) S(\displaystyle \sum_{i=1}^{h} c_i \sigma(w_{i}(x-s^{nn}t))+(2u_+-u_-),\\ 
V^{nn}(t, x)=3(v_+-v_-) S(\displaystyle\sum_{i=h+1}^{2h} c_i \sigma(w_{i}(x-s^{nn}t))+(2v_--v_+),
\end{cases} \nonumber 
\end{equation}
such that,
\begin{align*}
    Loss_{Total}<\epsilon.
\end{align*}
\begin{proof}
    Let $U$ be a solution of \eqref{3.1} that satisfies the boundary condition with the translation constraint $ U(0)=\frac{u_++u_-}{2}$. We first substitute the correct speed $s$ for $s^{nn}$ and denote $x-st$ by $z$. By Corollary \ref{Cor3.4}, there exists an approximation \begin{equation}U^{nn}(z)=3(u_--u_+) S(\sum_{i=0}^{h} c_i \sigma(w_{i}(z))+(2u_+-u_-) \nonumber \end{equation} of $U(z)$ such that $||D^{\alpha}(U)-D^{\alpha}(U^{nn})||_{L^{\infty}(K)}<\epsilon$, $\forall \alpha \in \left\{0, 1, 2, \cdots , m\right\}$ for a given small $\epsilon>0$. (Note that an approximation $V^{nn}$ of $V$ satisfying similar conditions exists as well.)
    \\ It is clear that $Loss_{Trans}\le \epsilon^2 +\epsilon^2=2\epsilon^2$ so that it can be bounded by a constant multiple of $\epsilon$. We then have 
    \begin{align*}
    Loss^{(1)}_{GE}&=||s^{nn}U^{nn}_{z}+\chi(UV)_{z}+DU_{zz}||^{2}_{L^2([-a, a])}
    \\&=||sU^{nn}_{z}+\chi(U^{nn}V^{nn})_{z}+DU^{nn}_{zz}||^{2}_{L^2([-a, a])}
    \\&=||s(U^{nn}-U)_{z}+\chi(U^{nn}V^{nn}-UV)_{z}+D(U^{nn}-U)_{zz}||^{2}_{L^2([-a, a])}
    \\&\le (||s(U^{nn}-U)_{z}||_{L^2([-a, a])} + ||\chi(U^{nn}V^{nn}-UV)_{z}||_{L^2([-a, a])} 
    \\&+ ||D(U^{nn}-U)_{zz}||_{L^2([-a, a])})^{2}.
    \end{align*}
    First and third term in the square on the right side are clearly bounded by a constant multiple of $\epsilon$.
    \\ For the second term, by a standard argument,
    \begin{align}
    &||(U^{nn}V^{nn}-UV)_{z}||_{L^2([-a, a])} \nonumber
    \\&\le ||U^{nn}(V^{nn}-V)||_{L^2([-a, a])} +||(U^{nn}-U)V||_{L^2([-a, a])} \nonumber
    \\&\le ||U^{nn}||_{L^{\infty}([-a, a])}||V^{nn}-V||_{L^2([-a, a])} +||V||_{L^{\infty}([-a, a])}||U^{nn}-U||_{L^2([-a, a])} \nonumber
    \\&\le C_{1}\epsilon + C_{2}\epsilon \nonumber,
    \end{align}
    where the last inequality holds since $U^{nn}$ and $V$ are bounded. $Loss^{(2)}_{GE}$ can be estimated in a similar way to obtain the same type of bound.
    \\Finally, let $U(-a), U(a), V(-a), V(a)= u_-+\eta_1(a), u_++\eta_2(a), v_-+\eta_3(a), v_++\eta_4(a) $. By the asymptotic behavior(or boundary) of solutions, we get the following eight limits associated with $Loss_{limit}$ and $Loss_{BC}$.
    \begin{align*}
    \displaystyle \lim_{a\rightarrow \pm\infty} \eta_1(a), \lim_{a\rightarrow \pm\infty} \eta_2(a), \lim_{a\rightarrow \pm\infty} \eta_3(a), \lim_{a\rightarrow \pm\infty} \eta_4(a)=0,
    \\\displaystyle \lim_{a\rightarrow \pm\infty} \eta_1'(a), \lim_{a\rightarrow \pm\infty} \eta_2'(a), \lim_{a\rightarrow \pm\infty} \eta_3'(a), \lim_{a\rightarrow \pm\infty} \eta_4'(a)=0. 
    \end{align*}
    We finally derive the following estimates.
    \begin{align*}
    Loss_{BC}&=(\eta_1'(a)+(U^{nn}(-a)-U(-a))')^2+(\eta_2'(a)+(U^{nn}(a)-U(a))')^2+ \nonumber
    \\&(\eta_3'(a)+(V^{nn}(-a)-V(-a))')^2+(\eta_4'(a)+(V^{nn}(a)-V(a))')^2 \nonumber
    \\\le& 2(\displaystyle \sum_{i=1}^{4}(\eta_i'(a))^2+4\epsilon^2 ), \nonumber
    \\Loss_{Limit}^{(1)} &= (\eta_{1}(a)+U^{nn}-U)^2  + (\eta_{2}(a)+U^{nn}-U)^2  \nonumber
    \\&\le 2(2\epsilon^2+\eta_1(a)^2+\eta_2(a)^2).
    \end{align*} \nonumber
    Therefore, we obtain the desired result.
\end{proof}
\end{theorem}
\begin{remark} Denote $U^{nn}-U, V^{nn}-V$ by $E, F$ respectively. Due to the fact that the upper and lower bounds of $U^{nn}, V^{nn}, U, V$ are exactly specified, $|E|, |F|$ will be also bounded functions with upper bounds $3|u_{-}|, 3|v_{-}|$ respectively . 
\end{remark}
Training using Adam Optimizer aims to make the value of the loss function converge to zero. The following theorem states that when the $Loss_{Total}$ is reduced, the estimated speed will converge to the correct value.
\begin{theorem} \label{3.8} Assume that the natural network architecture is constructed as in Theorem \ref{3.6}. If either $U(a)\ne U(-a)$ or $V(a)\ne V(-a)$ holds. Then for any $\epsilon>0$, there exists $M$ such that 
\begin{align*}
 \forall a>M, \hspace{.5cm} \exists \eta(a)\hspace{.2cm} such\hspace{.2cm} that\hspace{.2cm} Loss_{Total}<\eta(a)\hspace{.5cm} implies \hspace{5mm} |s^{nn}-s|<\epsilon.
\end{align*}
\end{theorem}
\begin{proof}
If we write down the equation that approximations $U^{nn}$ and $V^{nn}$ satisfy,
\begin{align}
\begin{cases}\label{eq3.3}
  s^{nn}U^{nn}_z+\chi(U^{nn}V^{nn})_z+DU^{nn}_{zz}=A(z), \\ 
  s^{nn}V^{nn}_z-(\varepsilon (V^{nn})^2 -U^{nn})_z +\varepsilon V^{nn}_{zz}=B(z),
\end{cases}
\end{align}
\\where $||A||_{L^{2}([-a ,a])}, ||B||_{L^{2}([-a ,a])}< Loss_{Total}.$ Subtracting the equation \eqref{3.1} from the equation above and integrating it over $(-a,a)$, we derive the equation below.
\begin{align}
 \begin{cases}
  (s^{nn}-s)U+s^{nn}E+\chi(EV+UF+EF)+DE_{z}\bigg\rvert_{-a}^{a}=\displaystyle\int_{-a}^{a}A(z)dz,\\ 
  (s^{nn}-s)V+s^{nn}F-(\varepsilon (F^2+2FV) -E) +\varepsilon F_{z}\bigg\rvert_{-a}^{a}=\displaystyle\int_{-a}^{a}B(z)dz.
\end{cases} \nonumber
\end{align}
Suppose that $U(a)\ne U(-a)$ holds.\\ Let $U(-a), U(a), V(-a), V(a) = u_-+\eta_1(a), u_++\eta_2(a), v_-+\eta_3(a), v_++\eta_4(a)$. Using the boundness of $U, V$ and the H\"older's inequality $\int_{-a}^{a}A(z)dz\le \sqrt{2a}||A||_{L^{2}([-a, a])}$, we have
\begin{equation} \displaystyle s^{nn}= \frac{s(U(a)-U(-a))+O(\displaystyle \sum_{i=1}^{4}\eta_i(a)+\eta_1'(a)+\eta_2'(a)+\sqrt{2a Loss_{Total}})}{U(a)-U(-a)+\eta_2(a)-\eta_1(a)}.\nonumber\end{equation}
Using the zero value of $\lim_{a\rightarrow \infty}\eta_{i}(a)$ and $\lim_{a\rightarrow \infty}\eta_{i}'(a)$ with standard arguments, we can obtain the theorem in the case of $U(a)\ne U(-a)$. The other case when $V(a)\ne V(-a)$ can be handled similarly. Therefore the theorem is proved.
\end{proof}
\begin{remark} Above theorem implies that $s^{nn}$ must be included in the interval $[s-\epsilon, s+\epsilon]$ so that it is bounded if the loss was sufficently reduced with an appropriate large interval $[-a, a]$.

\end{remark}
To compare the two solutions that satisfy similar governing equations and initial conditions, we cite a theorem in \cite{hirsch2012differential} from which some useful estimates were obtained using Gronwall's inequality. Unlike the previous results, the inequality below takes into account cases where the two functions have different initial conditions. By applying Theorem \ref{3.11}, we found the upper bound for the difference between neural network solutions and actual solutions. The bound depends on the length of the cut area, the value at the end point, and the differential coefficient error.
\begin{theorem} Let $U\in\mathbb{R}\times \mathbb{R}^n$ be an open set containing $(0, X(0))$ and $(0, Y(0))$. Let $F,G:U\rightarrow \mathbb{R}$ be continuously differentiable and satisfy the following two conditions.\\
$(i) |F(t, X)-G(t, X)|\le \epsilon, \quad \forall (t, X)\in U.\\
(ii) F(t, X)$ is $K-Lipchitz$ continuous in $X$.
\\If $X(t),Y(t)$ are solutions of the equation $X'=F(t, X)$ and $Y'=G(t,Y)$ respectively, then, 
\begin {align*}
        |X(t)-Y(t)|\le (|X(0)-Y(0)|+\frac{\epsilon}{K})\exp(K|t|)-\frac{\epsilon}{K}.
\end{align*}
\end{theorem} 
\begin{theorem}\label{3.11} Assume that the neural network architecture is constructed as in Theorem \ref{3.6}. If we write $\sqrt{(U-U^{nn})^2(x)+(V-V^{nn})^2(x)}=E(x)$, then the following inequality holds.
\begin{align}
E(x)\le (E(-a)+\frac{\sqrt{\epsilon_{1}^{2}+\epsilon_{2}^{2}}}{K})\exp(K|x+a|)-\frac{\sqrt{\epsilon_{1}^{2}+\epsilon_{2}^{2}}}{K} ,\nonumber
\end{align}
where 
\begin{align}
\epsilon_1= &(|U_z -U^{nn}_z|+\frac{s+\chi v_{-}}{D}|U-U^{nn}|+\frac{3u_{-} -u_{+}}{D}|s-s^{nn}|, \nonumber \\&+\frac{\chi(2u_{-}-u_{+})}{D}|V-V^{nn}|)(-a)+\frac{1}{D}\int_{-a}^{x}|f|dz, \nonumber \\
\epsilon_2=&(|V_z -V^{nn}_z|+(\frac{s}{\varepsilon}+v_{+}-3v_{-})|V-V^{nn}|+\frac{1}{\varepsilon}|U-U^{nn}|, \nonumber\\
&+\frac{v_{+}-3v_{-}}{\varepsilon}|s-s^{nn}|)(-a)+\frac{1}{\varepsilon}\int_{-a}^{x}|g|dz, \nonumber \\
K=&\sqrt{(\frac{s+\chi v_{+}}{D})^2 +(\frac{\chi u_+}{D})^2+(\frac{1}{\varepsilon})^2 +(\frac{-s+2\varepsilon v_-}{\varepsilon})^2}. \nonumber
\end{align}
\end{theorem}
\begin{proof} By integrating equations \eqref{3.1} and \eqref{eq3.3} over $(-a,x)$ and rearranging the equations for derivative terms, the two coupled equtions are derived.
\begin{align}
U_z&=\frac{1}{D}(sU(-a)+\chi UV(-a)-sU-\chi UV)+U_z(-a),\nonumber\\ 
V_z &= \frac{1}{\varepsilon }(sV(-a)+U(-a)-sV-U)+V^2-V^2 (-a)+V_z(-a),\nonumber \\
U^{nn}_z &= \frac{1}{D}(\int_{-a}^{x}fdz+s^{nn}U^{nn}(-a)+\chi U^{nn}V^{nn}(-a)-sU^{nn}\nonumber\\
&-\chi U^{nn}V^{nn})+U^{nn}_z(-a),\nonumber\\ 
V^{nn}_z &= \frac{1}{\varepsilon }(\int_{-a}^{x}gdz+s^{nn}V^{nn}(-a)+U^{nn}(-a)-s^{nn}V^{nn}-U^{nn})\nonumber\\
&+(V^{nn})^2-(V^{nn})^2 (-a)+V^{nn}_z(-a).\nonumber
\end{align}
Applying usual triangular inequalities and boundedness, the following is derived.
\begin{align}
|sU-s^{nn}U^{nn}|&\le |s||U-U^{nn}|+|2u_- - u_+||s-s^{nn}|,\nonumber \\
|UV-U^{nn}V^{nn}|&\le |v_-||U-U^{nn}|+|2u_- - u_+||V-V^{nn}|,\nonumber \\
|sV-s^{nn}V^{nn}|&\le |s||V-V^{nn}|+|2v_- - v_+||s-s^{nn}|,\nonumber \\
|V^2 - (V^{nn})^2|&\le |3v_- -v_+||V-V^{nn}|. \nonumber
\end{align}
The Jacobian matrix of the system \eqref{3.1} is calculated as follows.
\begin{align}
    J(U, V)= \begin{bmatrix}
 -\frac{s+\chi V}{D}& -\frac{\chi U}{D}\\ 
 -\frac{1}{\varepsilon}& \frac{-s+2\varepsilon V}{\varepsilon} 
\end{bmatrix}.\nonumber
\end{align}
Combining the chain rule and the mean value theorem, it can be seen that the right hand side of \eqref{3.1} is a Lipschitz continuous function whose Lipshcitz constant is the supremum of the Frobenius norm of $J(U,V)$. 
\\By Theorem 3.9, we can derive the desired estimate above.
\end{proof}
\subsection{Experiments}
In this section, we provide numerical experiments of the Keller--Segel system with a small $\varepsilon$. In the conducted experiments, a five-layer neural network with one-dimensional output was used. Each hidden layer consists of 512 hidden units, using the hyperbolic tangent function as an activation function. The weights are initialized based on LeCun initialization provided by PyTorch as a default \cite{paszke2019pytorch}. The loss function was minimized by using the Adam optimizer with an initial learning rate of 1e-4 for the speed variable $s^{nn}$ and 1e-6 for network weights, and the learning rates are decreased by a factor of 0.9 for every 5000 epochs. To calculate $Loss_{GE}$ in \eqref{sec2.2}, which is an approximation of definite integrals, 201 points were randomly selected from the interval $[-a, a]$ for every epoch. In the training process, we used randomly sampled points, but when plotting the value of $Loss_{Total}$ in epoch, we used a fixed uniform grid to compute the integral. 

Figure \ref{KS0} contains shapes of the solutions, speed and $Loss_{Total}$ that change as learning progresses when $(\varepsilon, D, \chi)=(0,~ 2,~ 0.5)$ with $a=200$. The precise value of $s$ is obtained as $1$ using the equation \eqref{eq3.2}. In (A) and (B), stable function values near the boundary show results consistent with the theoretically revealed exponential decay. In Figure \ref{KS0}(C) and (D), the red and blue vertical lines show the moments when a dramatic change occurs in speed and $Loss_{Total}$ rapidly decreases, respectively. It can be seen that $s^{nn}$ is approaching the correct answer before the loss is sufficiently reduced. The experimental results are consistent with that sufficient loss of power ensures an accurate speed approximation as described in Theorem \ref{3.8}. We remark that the function values converge outside the boundary, as shown in the upper left of Figure \ref{KS0} (A) although Theorem \ref{3.11} cannot explain it. The value of the function may not be reasonably predicted in the region where training was not conducted.
\begin{figure}[h]
  \includegraphics[width=\textwidth]{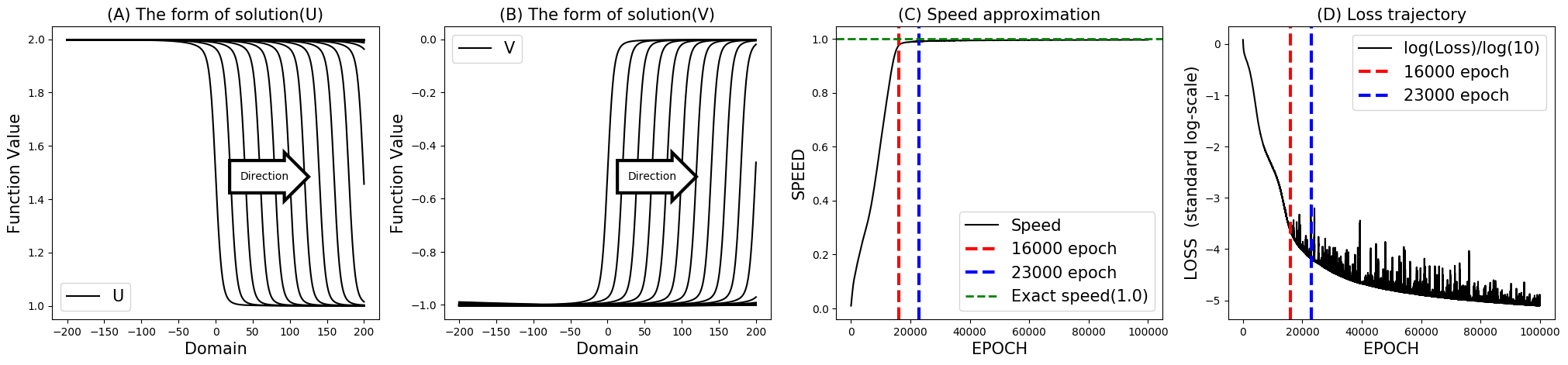}
  \caption{(A), (B): Approximated solutions for the Keller--Segel model, with model parameters $(\varepsilon, D, \chi)=(0,~2,~0.5)$, and the boundary conditions $(u_{-}, v_{-})=(2, -1)$ and $(u_{+}, v_{+})=(1, 0)$. (C): Estimated wave speeds in training epochs. (D): Trajectory of the total loss in training epochs.}
  \label{KS0}
\end{figure}
\par In the other experiment, the value of $\varepsilon$ was also set small enough to guarantee the existence and uniqueness of solutions. Given the value of $(\varepsilon, D, \chi)$ as $(0.1,~2,~0.9)$, the wave speed of the traveling wave solution is 0.9. In particular, (C) and (D) in Figure \ref{KS1} show that the moment when we significantly reduce the $Loss_{total}$ is almost identical to the moment when $s$ converges to the actual speed 0.9.
\begin{figure}[h]
\begin{center}
  \includegraphics[width=\linewidth]{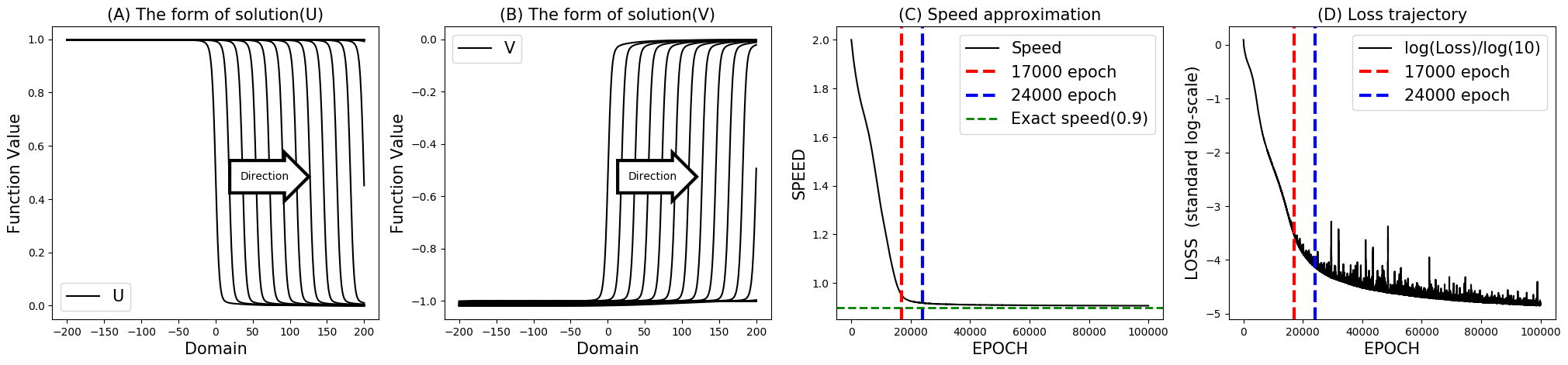}
  \caption{(A), (B): Approximated solutions for the Keller--Segel model, with model parameters $(\varepsilon, D, \chi)=(0.1,~2,~0.9)$, and the boundary conditions $(u_{-}, v_{-})=(1, -1)$ and $(u_{+}, v_{+})=(0, 0)$. (C): Estimated wave speeds in training epochs. (D): Trajectory of the total loss in training epochs.}
  \label{KS1}
\end{center}
\end{figure}
\par Figure \ref{KS0a} shows the effect of length of the interval on the convergence of losses. While maintaining the number of points used for calculating the integral, and varying the value of $a$ in $\{1, 10, 100, 200\}$, we compared the convergence process of $Loss_{Total}$ and the speed $s^{nn}$ during training. It was possible to learn the correct speed except for the case of using a small interval of length 2. However, according to Figure \ref{KS0a}(B), the value of the loss function converges to a value that is not sufficiently small for the interval [-10, 10]. The case of $a=100$ and the case of $a = 200$ showed a similar tendency in terms of the loss, while a slightly faster convergence was observed when the interval was shorter. Using the same number of points to approximate the integral with a similar computational cost seems to be the reason for the larger numerical error at wider intervals.

\begin{figure}[h]
  \includegraphics[width=\textwidth]{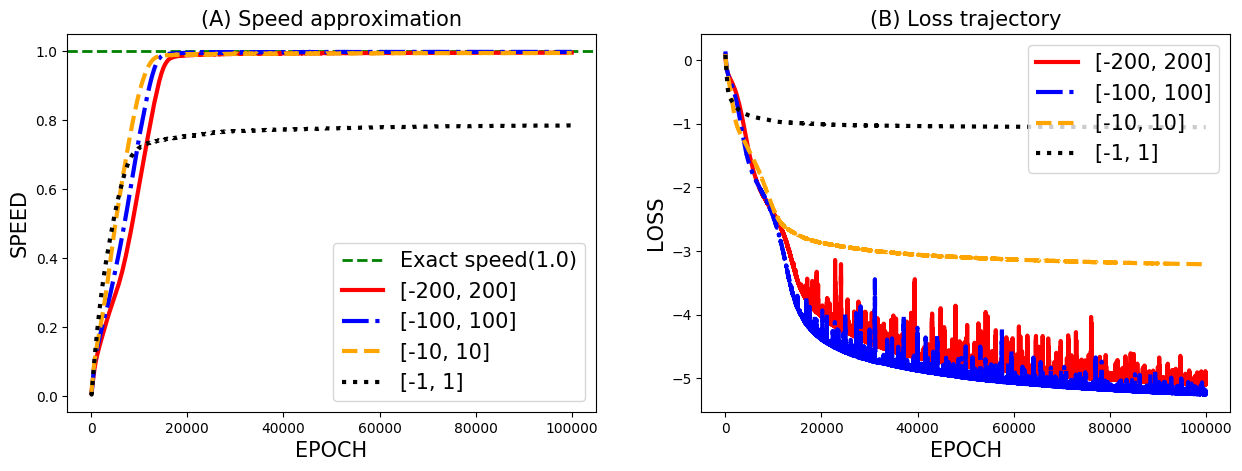}
  \caption{(A): Estimated wave speeds for different truncated intervals in training epochs. (B): Trajectories of the total loss for different truncated intervals. The model parameters are set to be $(\varepsilon, D, \chi)=(0,~2,~0.5)$, and the boundary condition is given by $(u_{-}, v_{-})=(2, -1)$ and $(u_{+}, v_{+})=(1, 0)$.}
  \label{KS0a}
\end{figure}

\par Based on the fact that the exact solution almost satisfies the Neumann boundary condition and the estimation in Theorem \ref{3.11}, we added $ Loss_ {BC} $ to the $Loss_{Total}$. Figure \ref{KS0BC} shows the comparison of experimental results with and without $Loss_{BC}$. We use a sigmoid function as an activation function for the output layer, so that the output is always positive. To avoid a situation where the Neumann boundary condition is satisfied before training, we mention that this experiment uses Xavier uniform initialization instead of LeCun initialization as the initial weight setting. As in Figure \ref{KS0BC}(A), the convergence of the speed was completed within a similar time. On the other hand, in Figure \ref{KS0BC}(B), we can see that the $Loss_{Total}$ is decreasing much faster when $Loss_{BC}$ is contained in the loss function.

\begin{figure}[h]
  \includegraphics[width=\textwidth]{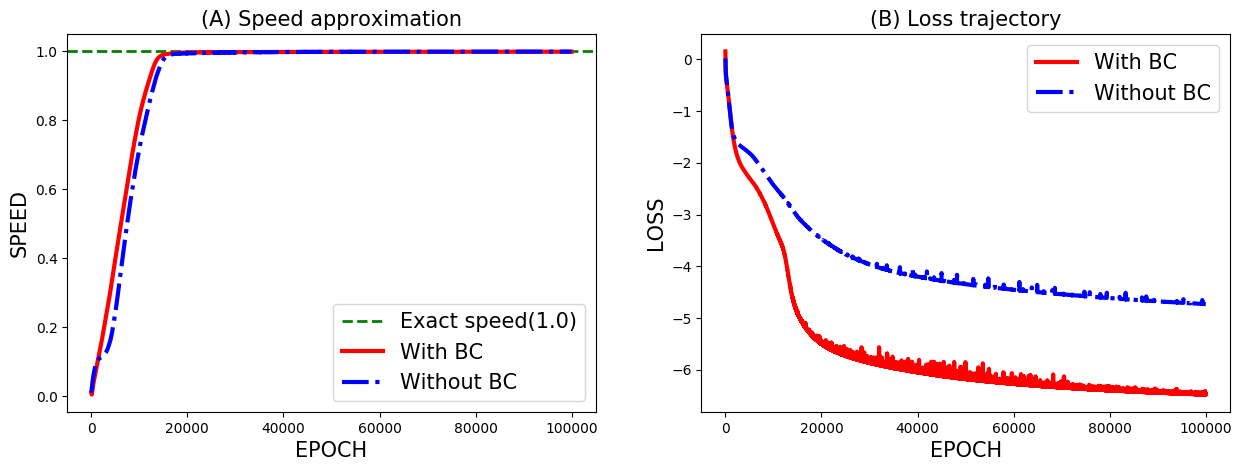}
  \caption{(A): Estimated wave speeds when trained with and without the boundary loss function. (B): Trajectories of the total loss when trained with and without the boundary loss. The model parameters are set to be $(\varepsilon, D, \chi)=(0,~2,~0.5)$, and the boundary condition is given by $(u_{-}, v_{-})=(2, -1)$ and $(u_{+}, v_{+})=(1, 0)$.}
  \label{KS0BC}
\end{figure}

\par With a few modifications, the equation \eqref{KSO} can be extended to the following multi-dimensional input problem in $\mathbb{R}^n$.
\begin{align}
 \begin{cases}
 u_t=\nabla \cdot(D\nabla u- \chi u c^{-1}\nabla c), \\ 
 c_t=\nabla \cdot (\epsilon \nabla c) -uc+\beta c. 
\end{cases}
\label{KSM}
\end{align}
\par The singularity term can be eliminated through a similar substitution $c^{-1}\nabla c$ $=(v_1, v_2, \cdots, v_n)$ as before. In this situation, the multi-dimensional traveling wave solution can be thought of as a function satisfying $(u, v)(x, t)=(U, V)$ $(k\cdot x - st)$, where $ v=\sum_{i=1}^{n} v_i$ and $k = \frac{1}{\sqrt{n}}(1, 1, \cdots, 1)$. In order to show an applicability to high dimensional problems, we conducted an experiment for the problem with 4 dimensional input. The domain $\mathbb{R}^4$ was truncated to $[-100,100]^4$, and $8^4$ randomly sampled points were used in each epoch to approximate the integral. Given the input values $t, x_1, x_2, x_3, x_4$, we used the characteristic layer as $\frac{1}{2}(x_1 + x_2 + x_3 + x_4)-s\cdot t$. For the exact solution, the function value should be determined by the value of $x_1 +x_2+x_3+x_4$. Figure \ref{KS0_4D1} shows that when two inputs are fixed and only the other two inputs change, the characteristic line of $U^{nn}$ with a slope of 1 is obtained. Figure \ref{KS0_4D2} shows similar convergence results for the wave speed and the total loss. We can also observe results compatible to the previous one that the convergence of the speed variable $s^{nn}$ precedes the optimization of the loss function. Overall, we observe that the proposed method can be used to approximate the traveling wave solution in higher dimensions.
\begin{figure}[h]
\begin{center}
  \includegraphics[width=\linewidth]{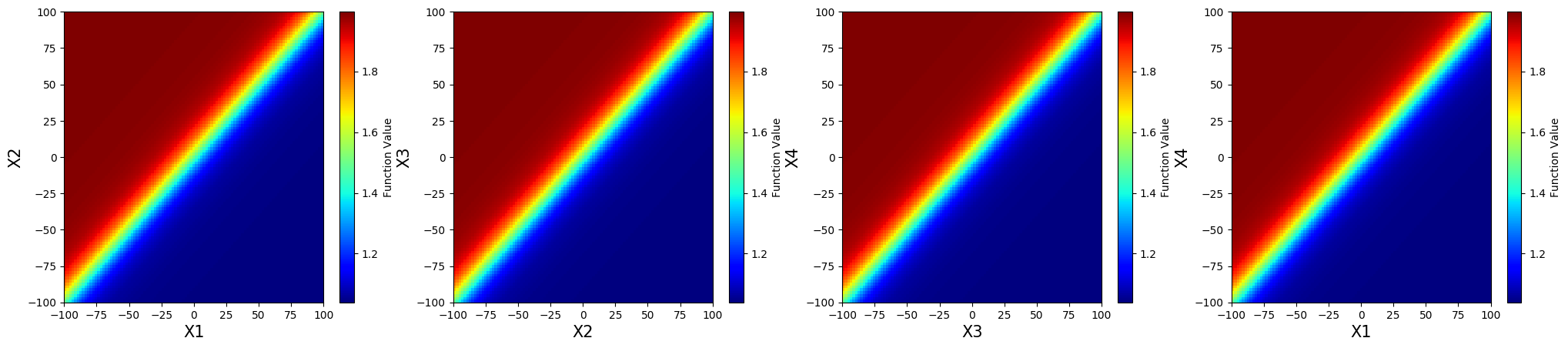}
  \caption{Approximated solutions for the 4-dimensional Keller--Segel model with the model paramters $(\varepsilon, D, \chi)=(0,~2,~0.5)$. To plot the 4-dimensional results, we fit the time at $t=0$, two of the four x-axis are fixed to be 0 and the values of the remaining two axes are sampled from -100 to 100.}
  \label{KS0_4D1}
\end{center}
\end{figure}

\begin{figure}[h]
\begin{center}
  \includegraphics[width=\linewidth]{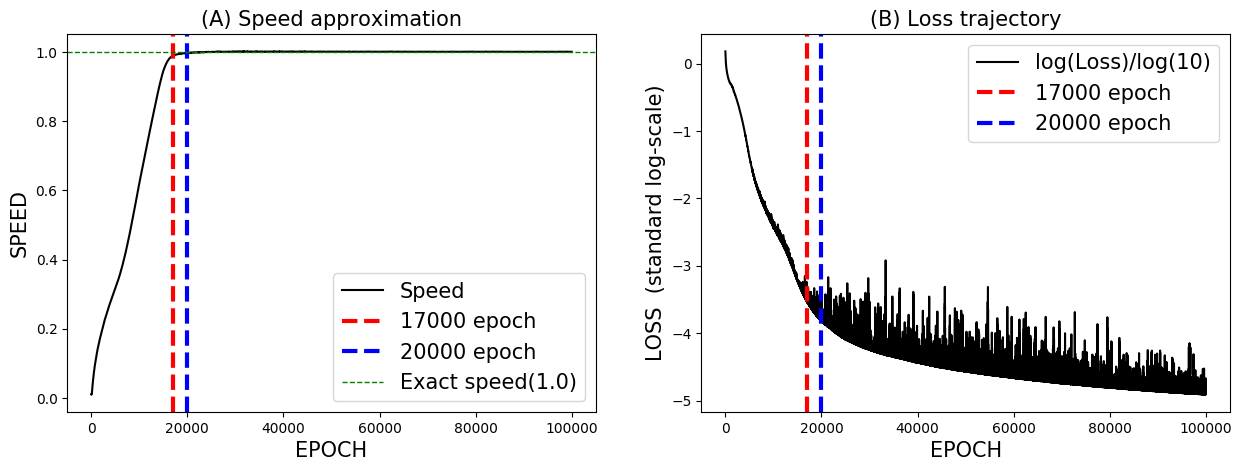}
  \caption{(A): Estimated wave speeds in training epochs for the multi-dimensional example. (B): Trajectory of the total loss in training epochs for the multi-dimensional example. The model paramters are set to be $(\varepsilon, D, \chi)=(0,~2,~0.5)$ and the boundary conditions are $(u_{-}, v_{-})=(2, -1)$ and $(u_{+}, v_{+})=(1, 0)$.}
  \label{KS0_4D2}
\end{center}
\end{figure}
\begin{remark}
After employing the traveling wave ansatz, the above multi-dimensional problem can be transformed into a one-dimensional ordinary differential equation with unknown coefficients. There are some cases where the dimension of the problem after substitution is larger than or identical to that of the original problem (See, \cite{xin1991existence}). We believe that our method can be directly applied to those by slightly modifying the ansatz layer.
\end{remark}

%Section4--------------------------------------------------------------------
\section{Appplications to Allen--Cahn Model with relaxation} \label{sec4}
In this section, we consider the Allen--Cahn model with relaxation which is written as below.
\begin{align}
\begin{cases}
 u_t=v_{x}+h(u),\\ 
 \tau v_t=u_{x}-v,
\end{cases}\nonumber
\end{align}
with the boundary conditions
\begin{align}
    (u(0,x),v(0,x))\rightarrow (0,0 )\text{ as } x\rightarrow -\infty, \nonumber \\ (u(0,x),v(0,x))\rightarrow (1,0) \text{ as }x\rightarrow +\infty. \nonumber
\end{align}
\par Here $\tau$, denoting the time-scale, is given as a nonnegative constant parameter. By imposing a traveling wave ansatz $(u, v)(x, t)=(U, V)(x-st)$, we obtain the following ordinary differential equations.
\begin{align}
\begin{cases} \label{eq4.3}
 sU'+V'+h(U)=0,\\ 
 U'+\tau sV'-V=0,
\end{cases}
\end{align}
with the boundary conditions
\begin{align}
   (U, V)(-\infty)=(0, 0), \quad (U, V)(+\infty)=(1, 0). \nonumber
\end{align}
\par Set $h(u)$ as $u(1-u)(u-\alpha)$, where $\alpha \in (0, 1)$. By combining Theorem 1.1, Proposition 2.1 in \cite{lattanzio2016analytical}, and minimum values of wave speeds in \cite{mendez1999speed}, we can organize the known facts about solutions and speed as below.
\begin{theorem} \label{4.1} If $\displaystyle \sup\limits_{u\in[0,1]}\tau h'(u)<1$ holds, then there exists a unique wave speed $s$ where the system with the asymptotic condition has a traveling wave solution $(U, V)$. Additionally, the following properties are established. 
\\(i) The function $U, V$ are positive and $U$ is monotone increasing.
\\(ii) $s$ has the same sign as $-\int_{0}^{1} h(u) du$
\\(iii) $\frac{\sqrt2(\alpha-\frac{1}{2})}{\sqrt{(1-\frac{1}{5}(1-2\alpha+2\alpha^2)\tau)^2+\frac{1}{2}\tau(1-2\alpha)^2}} \le s <\frac{1}{\sqrt{\tau}}$
\\(iv) For $\tau=0$, $s$ is explicitly given as $\sqrt{2}(\alpha-\frac{1}{2})$.
\end{theorem}
The following results are obtained by applying the methods used for the attestation in the previous section.

\begin{theorem}
Assume that the neural network architecture of $U$ is constructed as in Theorem \ref{3.6} and $V$ is constructed as in Theorem \ref{Thm3.3}. Let's $E(x)$ denote the function $\sqrt{(U-U^{nn})^2+(V-V^{nn})^2}(x)$. Then, the following inequality holds.
\begin{align}
E(x)\le (E(-a)+\frac{\sqrt{\epsilon_{1}^{2}+\epsilon_{2}^{2}}}{K})\exp(K|x+a|)-\frac{\sqrt{\epsilon_{1}^{2}+\epsilon_{2}^{2}}}{K},  \nonumber
\end{align}
where
\begin{align}
    \epsilon_1=&|\frac{1}{1-\tau s^2}-\frac{1}{1-\tau {(s^{nn})}^2}|+|\frac{\tau s}{1-\tau s^2}-\frac{\tau s^{nn}}{1-\tau {(s^{nn})}^2}|, \nonumber \\
    \epsilon_2=&|\frac{s}{1-\tau s^2 }-\frac{ s^{nn}}{1-\tau {(s^{nn})}^2}|+|\frac{1}{1-\tau s^2}-\frac{1}{1-\tau {(s^{nn})}^2}, \nonumber \\
    K=&\frac{1}{1-\tau s^2 }\sqrt{(\tau s \alpha)^2 +\alpha^2+s^2 +1}. \nonumber
\end{align}

\end{theorem}
\subsection{Loss Functions} Firstly, we note that the derivatives of the solution converges to zero so that we add further the Neumann boundary condition for a truncated domain.
Observing the asymptotic behavior of the solution of \eqref{eq4.3} and using the fact that $\tau s^2<1$ specified in Theorem \ref{4.1}, it can be confirmed that the extreme values of the derivatives become zero. Since only $U$ has monotonicity, $(U^{nn}(0)-\frac{u_-+u_+}{2})^2$ should be used as $Loss_{Trans}$. For $Loss_{GE}$, we used the usual $L^2$ error of the governing equation of the Allen--Cahn equation with a relaxation model. $Loss_{Limit}$ is constructed as in Section \ref{sec2} with $(u_{-}, v_{-})=(0, 0)$ and $(u_{+}, v_{+})=(1, 0)$. Before creating $Loss_{Total}$ by summing all four Losses, $Loss_{GE}$ was divided by $2L$, the length of truncated domain. Setting less weight to one loss causes the other losses to decrease first in the beginning period of learning. It was experimentally confirmed that learning the boundary conditions and extreme values first yield better results for the final approximation.
\begin{align}
    Loss_{Total}=\frac{1}{2L}Loss_{GE}+Loss_{Limit}+Loss_{BC}+Loss_{Trans}.\nonumber
\end{align}
\subsection{Numerical results}
The original domain, real line, was replaced by a finite interval $[-200,200]$ and learning was done only within it. The hyper-parameters such as a learning rate and a decay rate were set to be the same as in the experiments in the previous section. Both depth and the number of hidden units are the same as in the previous section. We used the hyperbolic tangent function as an activation function, and the weights are initialized by using LeCun initialization. The parameters $\tau$ and $\alpha$ were set from 0 to 3 and from 0.6 to 0.9 respectively to meet the prerequisites for Theorem \ref{4.1}. Figure \ref{AC} shows the trained solutions on $[-10, 10]$. It can be seen that most of the changes in values of solutions occur far from $z=-200$ and $z=200$ on which $Loss_{Limit}$ and $Loss_{BC}$ are defined.
\begin{figure}[h]
\includegraphics[width=\linewidth]{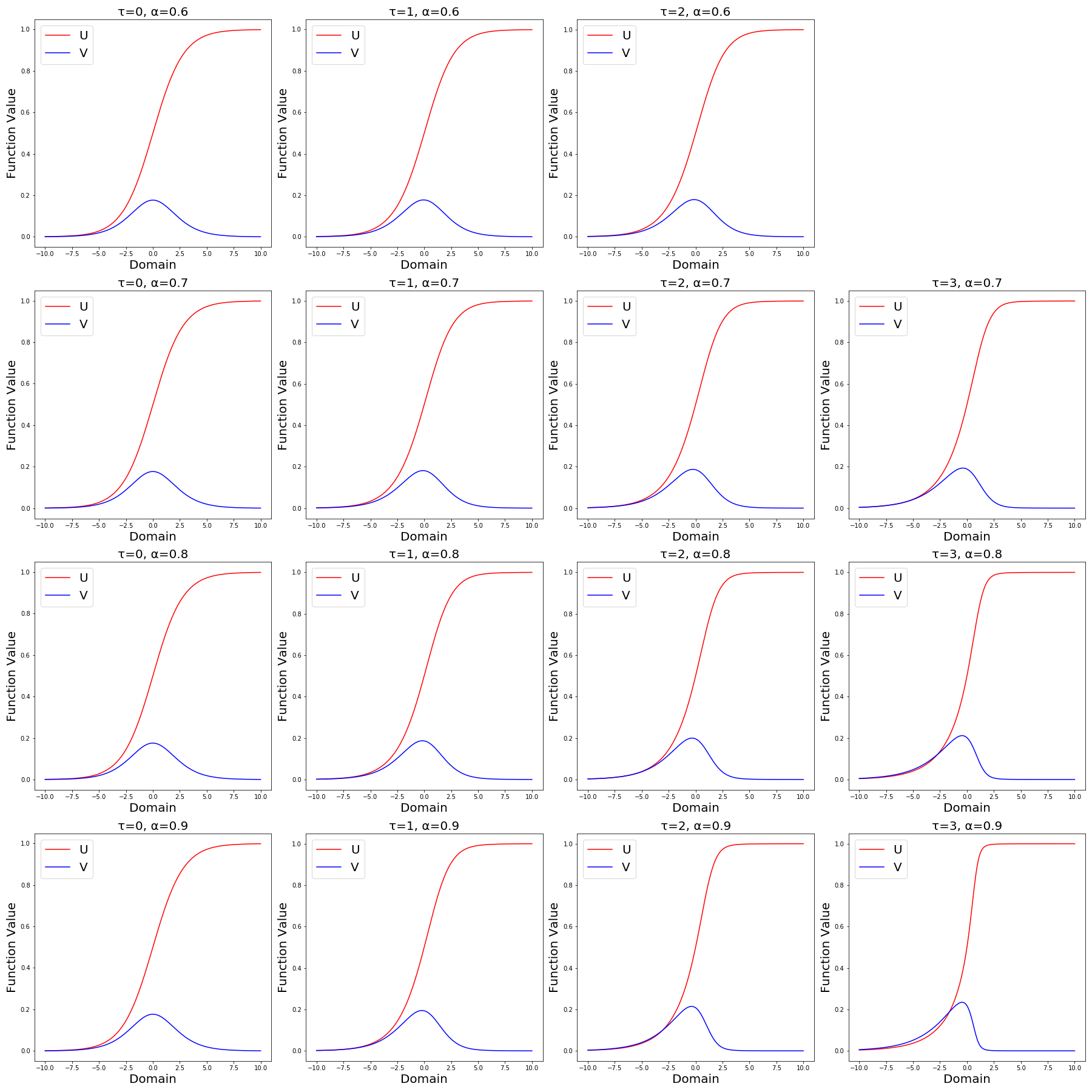}
\caption{Approximated solutions of the Allen--Cahn model with relaxation for different model parameters $(\tau, \alpha)$ with $\displaystyle  \sup\limits_{u\in[0,1]}\tau h'(u)<1$.}
\label{AC}
\end{figure}
\par If the value of $\tau$ is given as zero, the speed can be obtained explicitly as $\sqrt{2}(\alpha-\frac{1}{2})$ by Theorem \ref{4.1}. The fist line of Figure \ref{AC(exact)} shows whether the neural network solution predicts the correct speed for each alpha value. The graphs of the second line represent how losses have changed in the learning process, which is interpreted as having a pattern similar to changes in the learned speed. In Figure \ref{AC}, the approximation of $U$ was conducted satisfying the monotone increasing property. Additionally, It was implied that $V$, though not monotonous, continues to grow to a certain point near origin and then continues to decline. The slope at which $V$ decreases became steeper as the value of $\alpha$ or $\tau$ increased. The rapid change in the value of the function has had the effect of slowing the convergence of learning.   
\begin{figure}[h]
    \centering
    \includegraphics[width=\linewidth]{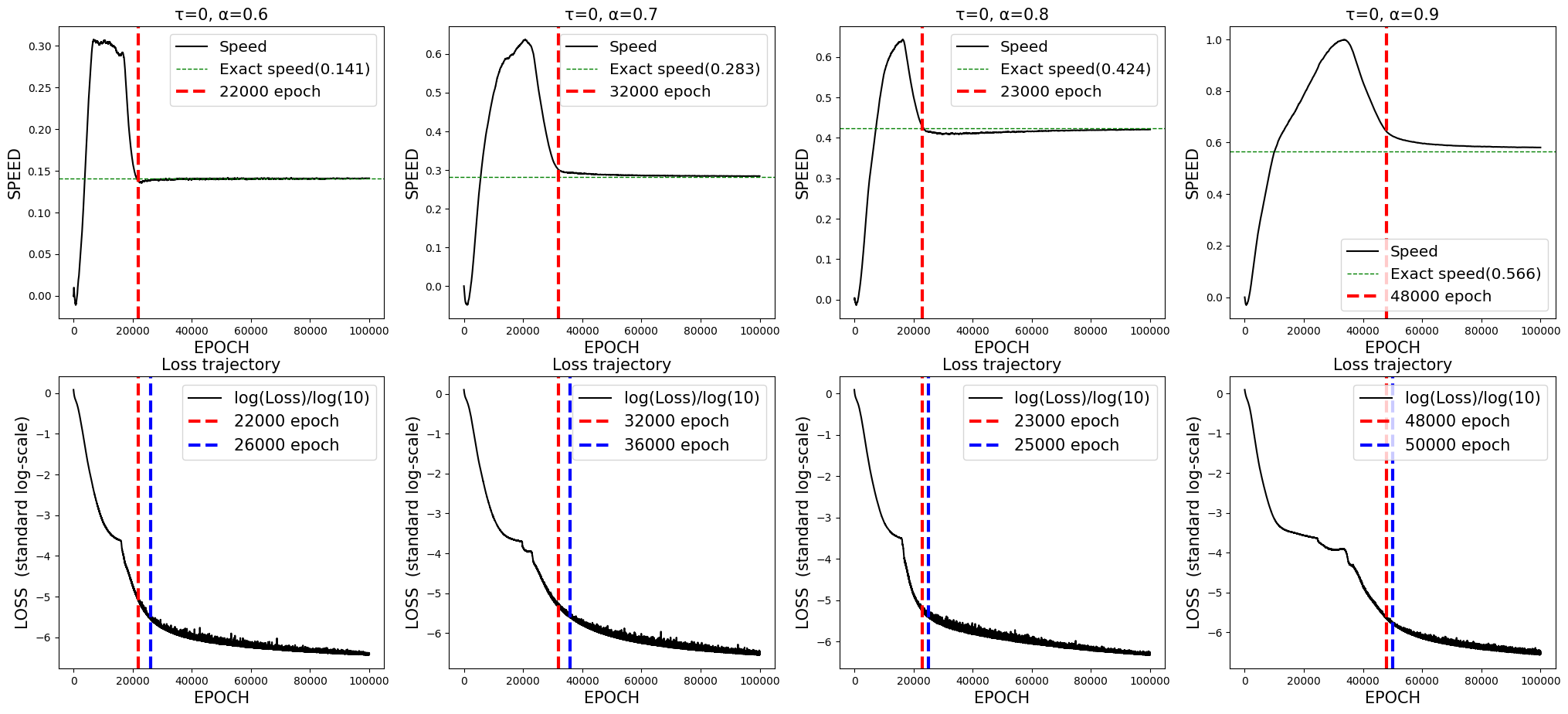}
    \caption{First row: Estimated wave speed in training epochs for different $\alpha$. Second row: Trajectories of the total loss in training epochs for different $\alpha$. In all cases $\tau$ is set to be 0, where the exact speed is known.}
    \label{AC(exact)}
\end{figure}
\par The lower and upper bounds of the speed mentioned in (iii) of Theorem \ref{4.1} are listed in Table \ref{min} and \ref{max}. The increasing values of alpha and tau narrow the gap between the two bounds, so it is suitable for verifying speed predictions accurately. Table \ref{n} and \ref{nn} contain estimated speed values for a given parameter using either the numerical method or our neural network method. In \cite{lattanzio2016analytical}, when the speed $s$ was considered as a variable, it was revealed that the point where the trajectory passing through $(U, V)=(0,0)$ and the straight line $U=\alpha$ meet decreases monotonically with respect to $s$. After showing a similar motonicity for trajectories through (1,0), the authors in \cite{lattanzio2016analytical} draw trajectory for each $s$ and record the values in Table \ref{n} that allow the two trajectories to intersect on the line $U=\alpha$. Observing the values in Table \ref{nn}, it was possible to accurately predict the speed when $\tau$ was 0. And also, even if $\tau$ is given as a different value, it predicted the values of speed similar to that of the numerical approximation.

\begin{table}[ht]
    \begin{minipage}[h]{0.48\linewidth}
        \centering % used for centering table
        \caption{The minimum bound} % title of Table
            \begin{tabular}{c c c c c c c c c} % centered columns (4 columns)
            \hline\hline %inserts double horizontal lines
            $\alpha$ & $\tau=0$ & 1 & 2 & 3\\ [0.5ex] % inserts table
            %heading
            \hline % inserts single horizontal line
            0.6 & {\bf 0.141} & 0.156 & 0.173 & (0.194)  \\ % inserting body of the table
            0.7 & {\bf 0.283} & 0.305 & 0.327 & 0.347  \\
            0.8 & {\bf 0.424} & 0.441 & 0.450 & 0.450 \\
            0.9 & {\bf 0.566} & 0.560 & 0.541 & 0.513 \\
             % [1ex] adds vertical space
            \hline %inserts single line
            \end{tabular}
        \label{min} % is used to refer this table in the text
    \end{minipage}
    \begin{minipage}[h]{0.48\linewidth}
        \centering % used for centering table
        \caption{The maximum bound} % title of Table
            \begin{tabular}{c c c c c c c c c} % centered columns (4 columns)
            \hline\hline %inserts double horizontal lines
            $\alpha$ & $\tau=0$ & 1 & 2 & 3 \\ [0.5ex] % inserts table
            %heading
            \hline % inserts single horizontal line
            0.6 & $\infty$ & 1.0 & 0.707 & (0.577) \\ % inserting body of the table
            0.7 & $\infty$ & 1.0 & 0.707 & 0.577 \\
            0.8 & $\infty$ & 1.0 & 0.707 & 0.577 \\
            0.9 & $\infty$ & 1.0 & 0.707 & 0.577 \\
             % [1ex] adds vertical space
            \hline %inserts single line
            \end{tabular}
        \label{max} % is used to refer this table in the text
    \end{minipage}
    
    \begin{minipage}[h]{0.48\linewidth}
        \centering
        \caption{Estimated speed(Numerical)} % title of Table
            \begin{tabular}{c c c c c c c c c} % centered columns (4 columns)
            \hline\hline %inserts double horizontal lines
            $\alpha$ & $\tau=0$ & 1 & 2 & 3 \\ [0.5ex] % inserts table
            %heading
            \hline % inserts single horizontal line
            0.6 & 0.14 & 0.16 & 0.17 & (0.20) \\ % inserting body of the table
            0.7 & 0.28 & 0.31 & 0.33 & 0.35 \\
            0.8 & 0.42 & 0.44 & 0.46 & 0.46 \\
            0.9 & 0.57 & 0.56 & 0.55 & 0.52 \\
             % [1ex] adds vertical space
            \hline %inserts single line
            \end{tabular}
        \label{n} % is used to refer this table in the text
    \end{minipage}
    \begin{minipage}[h]{0.48\linewidth}
        \centering % used for centering table
        \caption{Estimated speed(NN)} % title of Table
            \begin{tabular}{c c c c c c c c c} % centered columns (4 columns)
            \hline\hline %inserts double horizontal lines
            $\alpha$ & $\tau=0$ & 1 & 2 & 3 \\ [0.5ex] % inserts table
            %heading
            \hline % inserts single horizontal line
            0.6 & 0.141 & 0.156 & 0.173 & (-) \\ % inserting body of the table
            0.7 & 0.283 & 0.305 & 0.329 & 0.351 \\
            0.8 & 0.424 & 0.443 & 0.455 & 0.458 \\
            0.9 & 0.566 & 0.564 & 0.549 & 0.523 \\
             % [1ex] adds vertical space
            \hline %inserts single line
            \end{tabular}
        \label{nn} % is used to refer this table in the text
    \end{minipage}
\end{table}

\begin{figure}[h]
    \centering
    \includegraphics[width=\linewidth]{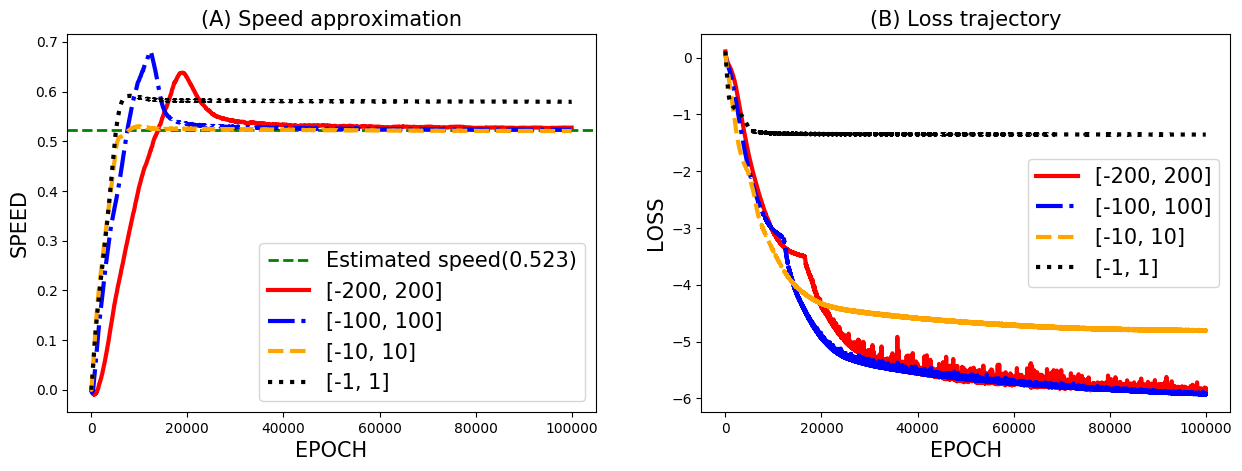}
    \caption{(A): Estimated wave speeds for different truncated intervals in training epochs. (B):  Trajectories of the total loss for different truncated intervals in training epochs. The model parameters are set to be $\tau$ = 0.9 and $\alpha$=3.}
    \label{AC15a}
\end{figure}
\par Experiments were conducted on how long the interval length should be to obtain a reasonable approximation of solutions for the Allen--Cahn equation with the relaxation model. It was intended to reveal whether the length of the interval can be determined even when the speed is unknown. We chose the case where $\alpha = 0.9$ and $\tau = 3$ where the difference of lower and upper bounds of the speed is the smallest. In Figure \ref{AC15a}, learning using the interval [-1,1] failed to converge, and learning using [-10, 10] showed some difficulty in reaching a sufficiently small loss. Due to the error occurring in the numerical integration, learning on the interval [-200, 200] showed a slightly slower progress than learning on [-100, 100].

%Section5--------------------------------------------------------------------

\section{Applications to the Lotka--Volterra Competition Model} \label{sec5}
In this section, we discuss the Lotka--Volterra Competition model with two species. 
\begin{align}
\begin{cases}
 u_t=u_{xx}+u(1-u-kv),\\ 
 v_t=dv_{xx}+bv(1-v-hu),
\end{cases} \nonumber
\end{align}
with the boundary conditions
\begin{align}
    (u(0,x),v(0,x))\rightarrow (0,1) \text{ as } x\rightarrow -\infty, \nonumber\\ (u(0,x),v(0,x))\rightarrow (1,0) \text{ as } x\rightarrow +\infty. \nonumber
\end{align}
$b, d$ denote the intrinsic growth rate and diffusion coefficient respectively. $h$ and $k$ represent inter-specific competition coefficients. All the parameters are given as positive with $\min\left\{h, k\right\}>1$. The details of derivation of the model can be found in \cite{murray2007mathematical}. As in the previous section, applying the traveling wave ansatz $(u, v)(x, t)=(U, V)(x-st)$, we can derive the following equation. 
\begin{align}\label{5.1}
\begin{cases}
 U''+sU'+U(1-U-kV)=0,\\ 
 dV''+sV'+bV(1-V-hU)=0,
\end{cases} 
\end{align}
with the boundary conditions
\begin{align}
    (U, V)(-\infty)=(0, 1), \quad (U, V)(+\infty)=(1, 0).\nonumber
\end{align}
The uniqueness and existence of this system are also proven in \cite{kan1995parameter}. As for speed, there is relatively less known information than other equations. By applying substitution and the uniqueness of solution, the parameter values of $(b,h,k,d)$ with standing wave solutions were obtained in \cite{guo2013sign}. They then found a sign of the wave speed using the fact that the wave speed has a monotone dependence on parameters or the terms in which parameters are combined. In summary, they are stated as follows.
\begin{theorem} (Theorem 2.1 in \cite{kan1995parameter})
Suppose that $\min \left\{h,k\right\}>1$ holds. Then, the problem has a strictly monotone solution $(U, V)$ with speed $s$. Moreover, if $(\tilde{U}, \tilde{V})$ with speed $\tilde{s}$ is another positive solution of the problem, then $\tilde{s}$ must be equal to $s$ and there exists a constant $l\in \mathbb{R}$ such that $(U,V)(z)=(\tilde{U},\tilde{V})(z+l)$. 
\end{theorem}

\begin{theorem} (Theorem 1.1 in \cite{guo2013sign})
Suppose that $\min \left\{h,k\right\}>1$ holds. For $a=d$, $s(b, h, k, d)$ has the same sign as $(k-h)$. In particular, $s=0$ when $h=k$.  
\end{theorem}

Analogous results to those in the previous section are also obtained similarly.
\begin{theorem}
Suppose that the neural network architecture is constructed as in Theorem \ref{3.6}. If we write $\sqrt{(U-U^{nn})^2(x)+(V-V^{nn})^2(x)}=E(x)$, then the following inequality holds.
\begin{align}
E(x)\le (E(-a)+\frac{\sqrt{\epsilon_{1}^{2}+\epsilon_{2}^{2}}}{K})\exp(K|x+a|)-\frac{\sqrt{\epsilon_{1}^{2}+\epsilon_{2}^{2}}}{K}  \nonumber
\end{align}
for $x\in[-a, a]$, where 
\begin{align}
    \epsilon_1=&(|U_z -U^{nn}_z|+(|s|+8a+2ka)|U-U^{nn}| \nonumber\\&+2ka|V-V^{nn}|+3|s-s^{nn}|)(-a)+\int_{-a}^{x}|f|dz, \nonumber\\
    \epsilon_2=&(|V_z -V^{nn}_z|+\frac{|s|+8ab+4abh}{D}|V-V^{nn}|+\frac{2abh}{D}|U-U^{nn}| \nonumber\\
    &+\frac{2}{D}|s-s^{nn}|)(-a)+\frac{1}{\varepsilon}\int_{-a}^{x}|g|dz, \nonumber \\
    K=&\sqrt{(s+2ab-2ak)^2+(2ak)^2+\frac{(2abh)^2}{d^2}+\frac{(|s|+2bh+2abh)^2}{d^2}}. \nonumber
\end{align}

\end{theorem}
\subsection{Loss Functions}
Note that adding the Neumann boundary condition doesn't cause a conflict with finding a solution in the Lotka--Volterra competition model. After multiplying the first equation of \eqref{5.1} by $U'$ and integrating it over $[p,q]$, we derive the following equation.
\begin{align*}
    &\frac{1}{2}((U'(q))^2-(U'(p))^2)+s\int_p^q (U')^2 dz +\frac{1}{2}(U^2(q)-U^2(p)) -\frac{1}{3}(U^3(q)-U^3(p))\\&-k\int_p^q  U'UV dz=0.
\end{align*}
The positive function $U'UV$ is bounded by $U'U$ so that it must be integrable. Therefore, $ \lim_{q\rightarrow \infty}\int_{p}^{q}U'UVdz$ exists and so does $\int_{p}^{q} (U')^2dz$ which is an increasing function in $q$. Finally, $U'(q)$ converges to some value which can only be zero. The discussion on the other side or $V$ is resolved in a similar way so that $Loss_{BC}$ is included without any problem. The solution $V$ has also the strict monotonicity, but we created $Loss_{trans}$ based on the value of $U(0)$. $Loss_{GE}$ represents the $L^2$-error of the governing equation for the Lotka--Volterra competition model in the same way as other equations. $Loss_{Limit}$ is designed as in Section \ref{sec2} with $(u_-, v_-)=(0, 1)$ and $(u_+, v_+)=(1, 0)$. As before, the importance of $Loss_{GE}$ has been diminished to increase the likelihood that the neural network solution can converge. The weights are given as in Section \ref{sec4}. 
\begin{align}
    Loss_{Total}=\frac{1}{2a}Loss_{GE}+Loss_{Limit}+Loss_{BC}+Loss_{Trans}\nonumber
\end{align}
\subsection{Experiments}
As far as we know, the only known fact about speed in the Lotka--Volterra competition model is the sign. The first experiment was aimed at the approximation of the standing waveform, the only case in which the exact speed was known. The training was conducted on the truncated domain $[-200,200]$ using the Adam optimizer. The number of hidden layers, the number of hidden units, the activation function, and the weight initialization are the same as in the previous section. Initial learning rates were set to $2\cdot10^{-4}$ for the speed variable and $2\cdot10^{-6}$ for the network weights, respectively. For every 5000 epochs, the learning rates are decreased by a factor of 0.9. In Figure \ref{LV(zero)}(A), the color gradually turning blue from left to right suggests that our algorithm captures the monotonicity of the solution $U$. Similar results are observed in Figure \ref{LV(zero)}(B). Comparing (C) and (D) in Figure \ref{LV(zero)}, it was once again confirmed that a reduced loss ensures an accurate estimated speed.
\begin{figure}[h]
    \centering
    \includegraphics[width=\linewidth]{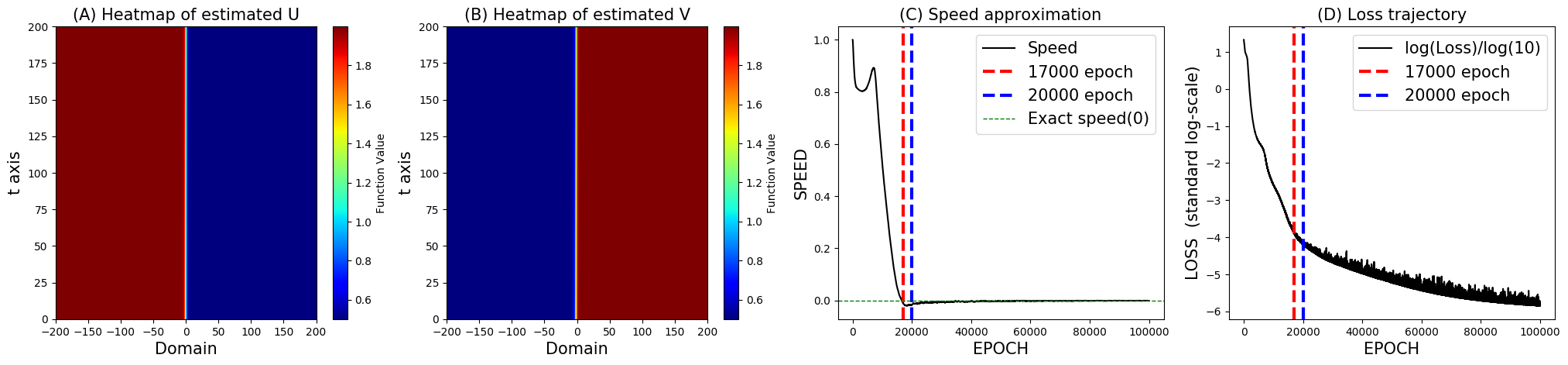}
    \caption{(A), (B): Approximated solutions for the Lotka--Volterra competition model with the model parameters are set to be $(a, h, k, d)=(2, 2, 2, 2)$, where the speed is exactly zero. (C): Estimated wave speeds in training epochs. (D): Trajectory of the total loss in training epochs.}
    \label{LV(zero)}
\end{figure}
\begin{figure}[h]
    \centering
    \includegraphics[width=\linewidth]{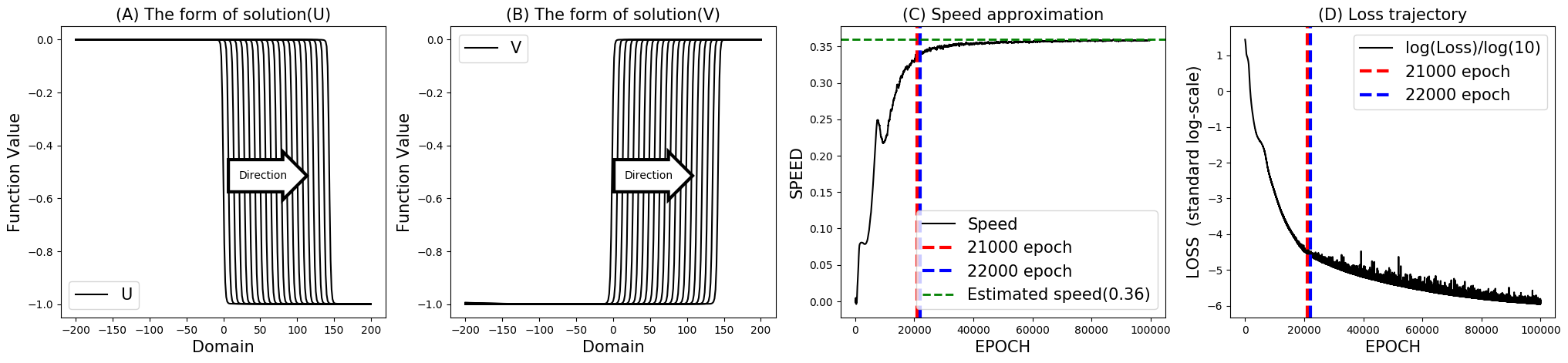}
    \caption{(A), (B): Approximated solutions for the Lotka--Volterra competition model, with model parameters $(a, h, k, d)=(2, 2, 3, 2)$ where the exact wave speed has a positive sign. (C): Estimated wave speeds in training epochs. (D): Trajectory of the total loss in training epochs.}
    \label{LV(pos)}
\end{figure}
\begin{figure}[h]
    \centering
    \includegraphics[width=\linewidth]{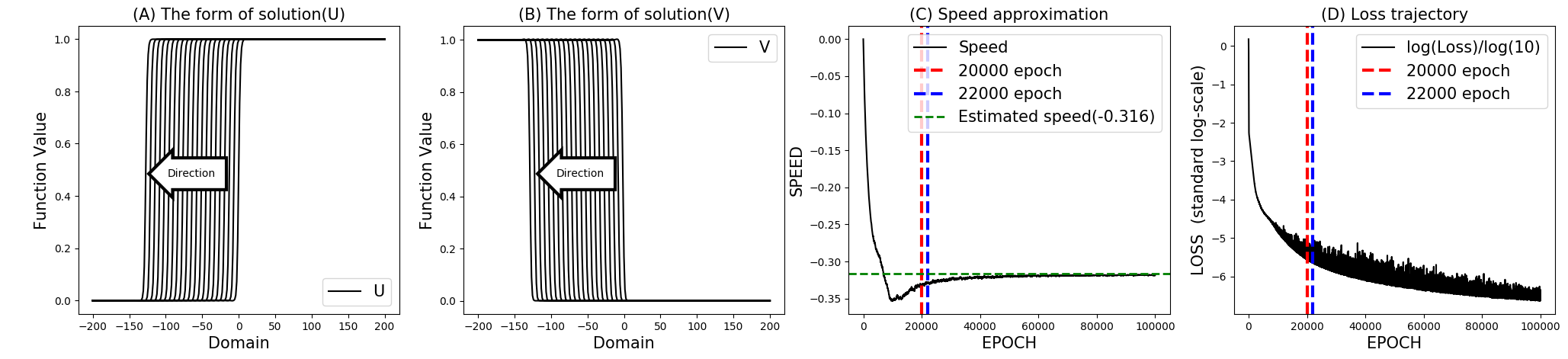}
    \caption{(A), (B): Approximated solutions for the Lotka--Volterra competition model, with model parameters $(a, h, k, d)=(2, 3, 2, 2)$ where the exact wave speed has a negative sign. (C): Estimated wave speeds in training epochs. (D): Trajectory of the total loss in training epochs.}
    \label{LV(neg)}
\end{figure}
\par The training processes and results of cases where only signs are known about speed are shown in Figure \ref{LV(pos)} and \ref{LV(neg)}. The initial value of $s^{nn}$ was set to zero to exclude prior knowledge of the sign. (A) and (B) indicate that the trained solution captures the monotonicity of U and V while accurately predicting the direction of wave speed. Figure \ref{LV(pos)}(C) and Figure \ref{LV(neg)}(C) show that the convergence of speed was almost completed before 100,000 epochs. Observing the graphs of \ref{LV(pos)}(D) and \ref{LV(neg)}(D), it can be confirmed that the convergence of the velocity and the convergence of the loss function occur simultaneously.

\par As mentioned earlier, nothing is known about the speed of the solution of \eqref{5.1} except for the sign. We compare the estimation results for different intervals to the example in Figure \ref{LV(pos)}, where the speed was estimated to be 0.36. In the case of the interval [-1, 1], the speed variable fails to converge, and the $Loss_{Total}$ is not sufficiently decreased. We also observed that increasing the length of the interval makes the training more accurate. One notable point is that Figure \ref{LV1a} shows that training on the interval [-10, 10] can yield solutions with a compatible loss and a faster convergence of estimated speed.
\begin{figure}[h]
    \centering
    \includegraphics[width=\linewidth]{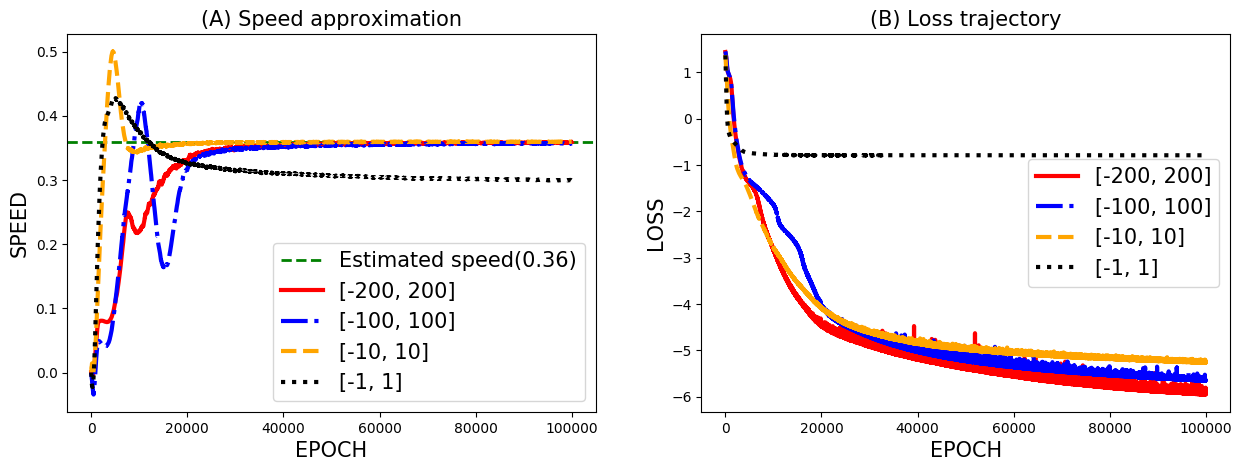}
    \caption{(A): Estimated wave speeds for different truncated intervals in training epochs. (B):  Trajectories of the total loss for different truncated intervals in training epochs. The model parameters are set to be $(a, h, k, d)=(2, 2, 3, 2)$ where the speed has a positive sign.}
    \label{LV1a}
\end{figure}

\par Figure \ref{LV1BC} shows a significant difference compared to Figure \ref{KS0BC}, between the cases whether the $Loss_{BC}$ is involved in the training. As we can see in the figure, the convergence speeds of both speed variable and the total loss are much faster when we train with the $Loss_{BC}$.

\begin{figure}[h]
    \centering
    \includegraphics[width=\linewidth]{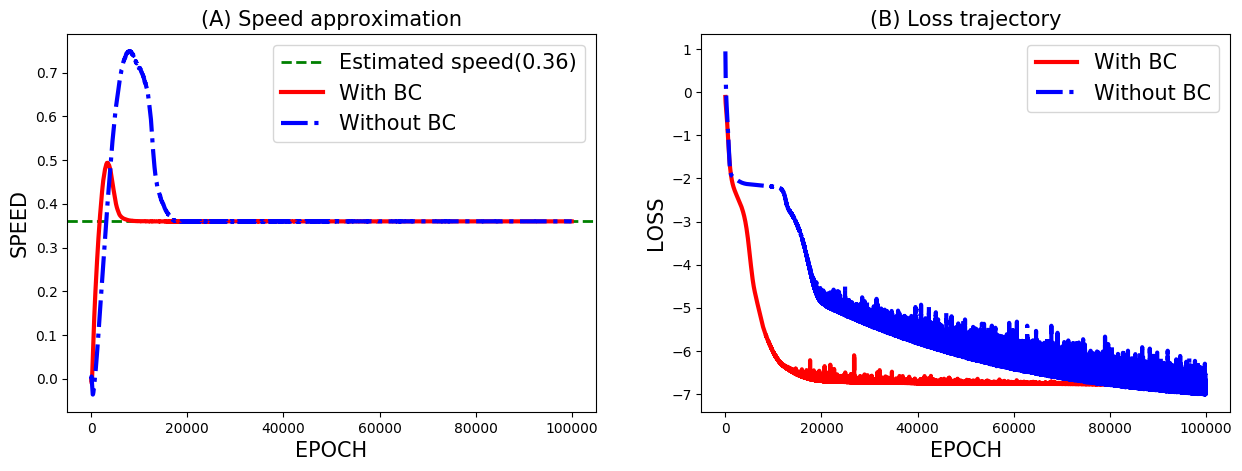}
    \caption{(A): Estimated wave speeds when trained with and without the boundary loss function for the Lotka--Volterra competition model. (B): Trajectories of the total loss when trained with and without the boundary loss for the Lotka--Volterra competition model. The model parameters are set to be $(a, h, k, d)=(2, 2, 3, 2)$ where the speed has a positive sign.}
    \label{LV1BC}
\end{figure}

%Section6--------------------------------------------------------------------
\section{Conclusion and future work} \label{sec6}
It is difficult to deal with the domain $\mathbb{R}$ numerically since it is unbounded. In order to overcome this, we truncated the real line to a bounded interval with a sufficiently large length. Moreover, to accurately approximate the solution, we added the Neumann boundary condition at the boundary of the truncated region that the solution asymptotically satisfies. However, the boundary condition we gave inherently possesses a small error due to the truncation. We leave a more thorough analysis for the treatment of this error term as a future work. 

\par Each of the equations covered in this paper was known to have a unique solution and the solutions are widely studied. Thus, some analytic properties of the solutions, such as monotonicity, can be considered as criteria for determining whether the solution is well approximated. 

\par On the other hand, there are many cases where uniqueness of the solution is not guaranteed, because there are multiple wave speeds that guarantee a solution as in the case of having a minimum wave speed. Even in this case, our neural network model provides only one solution as a correct answer. We believe that it is worthwhile to consider which one of the numerous solutions has approximated by a neural network. Furthermore, a novel way of approximating all possible solutions with all possible speeds should be devised.

\par Learning an equation solver that maps a set of model parameters to a solution is also necessary. For instance, in the Keller--Segel equation one should train a new neural network every time the values of $D, \chi, \epsilon$ change. As it takes a lot of time to learn the solution for each given model parameter, a further research on neural network methods that can rapidly predict the solutions even for a set of unseen model parameters would improve learning efficiency.

\bibliographystyle{amsplaindoi} %AMS references format with number references, and clickable URL links
\bibliography{bibliograph}

%%% input the BBL file when done using BibTeX \left(change the name to be the document name\right)
%\input{RBELinfty.bbl}

%%%%%%%%%%%%%%%%%%%%%%%%%%%%%%%%%%%%%%%%%%%%%%%%%%%%%%%%%%%%%%%%%%%%%%%%%%%%%%%%%%

%%% copy and paste the BBL file below here when you are ready to upload, and comment the above bibtex library

%%\newpage

\end{document}